\documentclass[11pt]{amsart}

\title[Water waves with compact vorticity]{Traveling water waves with compactly supported vorticity}
\author[Shatah]{Jalal Shatah$^\dagger$} 
\address{${}^\dagger$Courant Institute of Mathematical Sciences \\ New York University \\ New York, N.Y. 10012}
\email[J. Shatah]{shatah@cims.nyu.edu}

\author[Walsh]{Samuel Walsh$^\dagger$}
\email[S. Walsh]{walsh@cims.nyu.edu}

\author[Zeng]{Chongchun Zeng$^{*}$}
\address{$^{*}$School of Mathematics\\
Georgia Institute of Technology\\
Atlanta, GA 30332}
\email[C. Zengh]{zengch@math.gatech.edu}
\thanks{$^{*}$ The third author is funded in part by NSF DMS 1101423 and a part of the work was completed while visiting Courant Institute, New York University in 2011 and IMA, University of Minnesota in 2012.}

\usepackage{amssymb,stmaryrd}
\usepackage{amscd}
 \usepackage{mathrsfs,dsfont}
\usepackage[dvips]{graphicx}

\usepackage{geometry}                
\newcommand{\be}{\begin{equation} }
\newcommand{\ee}{\end{equation}}
\newcommand{\bse}{\begin{subequations}}
\newcommand{\ese}{\end{subequations}}

\newcommand{\jbracket}[1]{\langle{#1}\rangle}
\newcommand{\realpart}[1]{\operatorname{Re}{#1}}
\newcommand{\imagpart}[1]{\operatorname{Im}{#1}}

\newcommand{\supp}[1]{\operatorname{supp}{#1}}

\newcommand{\p}{\partial}
\newcommand{\BBS}{\mathbb{S}}
\newcommand{\BBR}{\mathbb{R}}
\newcommand{\CG}{\mathbf{G}}
\newcommand{\ep}{\epsilon}
\newcommand{\CH}{\mathcal{H}}

\newcommand{\mbv}{\mathbf{v}}
\newcommand{\mbe}{\mathbf{e}}
\newcommand{\CF}{\mathcal{F}}
\newcommand{\CX}{\mathcal{X}}
\newcommand{\CY}{\mathcal{Y}}

\theoremstyle{plain} 
\newtheorem{theorem}{Theorem}[section]
 
\newtheorem*{main}{Main~Theorem}
\newtheorem{lemma}{Lemma}[section]
\newtheorem{proposition}{Proposition}[section] 
\theoremstyle{definition}
\newtheorem{definition}{Definition}[section] 
\theoremstyle{definition} 
 
\theoremstyle{remark} 
\newtheorem{remark}{Remark}[section]

\numberwithin{equation}{section}

\usepackage{hyperref}
\hypersetup{ 
colorlinks=true, 
} 

\begin{document}
\maketitle

\begin{abstract}  In this paper, we prove the existence of two-dimensional, traveling, capillary-gravity, water waves with compactly supported vorticity.  Specifically, we consider the cases where the vorticity is a $\delta$-function (a point vortex), or has small compact support (a vortex patch).  Using a global bifurcation theoretic argument, we construct a continuum of finite-amplitude, finite-vorticity solutions for the periodic point vortex problem.   For the non-periodic case, with either a vortex point or patch, we prove the existence of a continuum of small-amplitude, small-vorticity solutions.
\end{abstract}

\section{Introduction} \label{intro:section}

Consider the water wave problem of infinite depth modeled by the free boundary problem of the incompressible Euler equation 
\begin{subequations}
\be \label{intro:Euler}
\p_t \mbv + (\mbv \cdot \nabla) \mbv + \nabla p + g \mbe_2=0, \qquad 
\ee
\be 
\nabla \cdot \mbv =0
\ee
\end{subequations}
on a moving domain 
\[ 
\Omega_t  := \{ (x_1, x_2) \in \mathbb{R}^2 : x_2 < 1+\eta(t, x_1)\}
\]
for some profile function $\eta(t, x_1)$, with the boundary conditions on the free surface $\p \Omega_t$
\begin{subequations}
\be \label{intro:kinematic}
\eta_t = -\eta' v_1 + v_2
\ee
\be \label{intro:tension}
p= \alpha^2 \kappa, \quad \alpha >0
\ee 
\end{subequations}
where $\kappa= \kappa(x_1)$ is the mean curvature of the surface
\[
\kappa (x_1) = - \frac {\eta''}{\big(1+(\eta')^2\big)^{\frac 32}}.
\]  
Here $- g \mbe_2 = (0, -g)^T$ is the constant gravitational acceleration. The kinematic boundary condition \eqref{intro:kinematic} is the requirement that the normal velocity of the boundary  agree with that of the fluid and condition \eqref{intro:tension} means that we include the surface tension in our consideration. 

Our goal in this paper is to seek traveling water waves in the form of 
\be \label{intro:travelingW}
\mbv= \mbv (x_1 - ct, x_2), \quad \eta = \eta(x_1 - ct), \quad \Omega = \{ (x_1, x_2) \in \mathbb{R}^2 : x_2 < 1+\eta(x_1)\}  
\ee
with {\it compactly supported vorticity}. We shall consider both cases
\begin{enumerate} 
\item [(Loc)] $\eta$ and $\mbv$ decay at infinity and 
\item [(Per)] $\eta$ and $\mbv$ are $2\pi L$-periodic in $x_1$, i.e. $x_1 \in L \BBS^1$ and without loss of generality 
\be \label{intro:eta}
\int_{-\pi L}^{\pi L} \eta \, dx_1 =0.
\ee
\end{enumerate}
Since the arguments for both of these two cases follow a similar procedure, we will focus on the localized case (Loc) and then give outlines for the case $x_1 \in \BBS^1$.

The overwhelming majority of the research on water waves has been done in the irrotational setting, both for mathematical convenience and on physical grounds.  The source of the convenience is clear enough: for velocity fields with gradient structure, the dynamics in the bulk of the fluid are captured simply by Laplace's equation.  This allows the problem to be pushed to the boundary, where it can be recast in a number of ways as a nonlocal equation (e.g,  as an integral equation following Nekrasov \cite{levicivita1925determination,nekrasov1951exact}, or using Dirichlet-to-Neumann operators as in Babenko \cite{babenko1987localexistence}.)  Moreover, because Kelvin's circulation theorem states that  an initially irrotational flow will remain irrotational absent external forcing, these considerations are physical.  Today, the existence theory for steady water waves is very well-developed (cf., e.g., surveys in \cite{groves2004steady,okamoto2001mathematical,strauss2010steady,toland1996stokes}), though some important open problems still remain. 

Yet rotational waves are completely commonplace in nature.  Indeed, the presence of wind forcing, temperature gradients, or even a slight heterogeneity in the density, generates vorticity.  The  study of rotational steady waves essentially begins with Dubreil-Jacotin in 1934 (cf. \cite{dubreil1934determination}), but the entire topic  lay dormant until the relatively recently.  The main breakthrough came in 2004, when Constantin and Strauss developed a systematic existence theory for two-dimensional, periodic, finite-depth, traveling gravity waves (cf. \cite{constantin2004exact}). Both the work of Constantin--Strauss and Dubreil-Jacotin begin with a simple observation:  If there are no stagnation points in the flow, then one can use the stream function as a vertical coordinate to fix the domain.   Doing so, one ultimately arrives at a quasilinear elliptic PDE on a strip (the quasilinearity coming from the change of coordinates), with a nonlinear boundary condition.  Shear flows, i.e., those where the free surface and all of the streamlines are flat, can be easily described in the new coordinates, and so one can build small-amplitude solutions by a perturbative argument.  Then, using a degree theoretic continuation method, Constantin and Strauss were able to obtain finite amplitude solutions.  Following the publication of \cite{constantin2004exact}, many authors have been able to generalize this approach, and we now have a bounty of analogous results for other physical regimes (cf., e.g., \cite{wahlen2006capgrav,hur2006global,walsh2009stratified,walsh2012wind}), for a class of weak solutions (cf. \cite{constantin2011discontinuous}), and even for some types of waves with stagnation points (cf. \cite{wahlen2009critical,ehrnstrom2010interior,ehrnstrom2010multiple}); a recent account of the field is given in \cite{strauss2010steady}.  

Though they differ in many essential details, all of these works rely on the fact that the vorticity is constant along streamlines for steady flows away from stagnation.  Indeed, even the papers dealing with the stagnant case, where this does not follow from physical principles, impose it \emph{a priori}.  Another common feature is that they all use shear flows as their class of trivial solutions for the bifurcation argument.  Taken together, these two facets make it impossible to construct waves with localized vorticity.   To see why conceptually, simply note that in a shear flow, every point sits on a unique streamline extending to $\pm \infty$ in the direction of wave propagation. Along that streamline, the vorticity is constant, and hence it cannot be localized unless it happens to vanish identically.  Naturally, small amplitude perturbations of these flows will share this property.  Indeed, if there are no stagnation points, \emph{all} of the streamlines are unbounded.  
 With stagnation, it is possible that some of the streamlines become closed upon bifurcation --- a phenomenon referred to as \emph{cats' eyes} --- but some unbounded streamlines will necessarily persist.  

Thus a vast gulf exists between the well-studied  irrotational steady waves, where of course the vorticity vanishes identically, and the current literature on rotational waves, where it cannot be allowed to vanish at infinity.    The objective of the present paper is to address this gap, and, in a sense, our approach weds the two outlined above.  We will be able to describe traveling waves where point vortices or eddies are suspended in an otherwise irrotational flow.  In particular, the traveling waves with compactly supported vortex patch found in this paper have finite energy.

The \emph{vorticity} $\omega$ of a 2-d velocity field $\mathbf{v}$ 
is defined to be the distribution
\be 
\omega := \partial_{x_1} v_2 - \partial_{x_2} v_1. \label{intro:defomega} 
\ee
When $\omega$ is a finite measure, we may define the \emph{vortex strength} to be $\epsilon = \omega(\Omega)$, and in particular, 
\be 
\epsilon := \int_{\Omega} \omega \, dx, \quad \text{ if } \quad \omega \in L^1(\Omega).\label{intro:defepsilon} 
\ee 
In the interior of the fluid, the incompressible Euler equation can be expressed in its vorticity formulation which takes the following form for traveling waves 
\be 
-c \partial_{x_1} \omega + \nabla \cdot (\omega \mathbf{v} )= (-c \mbe_1 + \mbv) \cdot \nabla \omega = 0. \label{intro:vorticityeq} 
\ee
It means that the vorticity is transported by the flow. We consider two problems 
\begin{enumerate}
\item [(PtV)] $\omega (x) = \ep \delta(x)$, a point mass away from the fluid boundary, without loss of generality, whose concentration point is taken as the origin.
\item [(VPa)] $D:=\supp(\omega) \subset \subset \Omega$ and $\omega$ is smooth on $D$ which is near the origin. 
\end{enumerate}
In fact, we consider both the localized (Loc) and periodic (Per) cases for the point vortex problem (PtV) and only the localized case (LoC) for the vortex patch problem (VPa). The periodic case (Per) can also be considered for the vortex patch problem (VPa), but the computation is too tedious and we simply skip it in this manuscript. Vorticity equation \eqref{intro:vorticityeq} can be interpreted in the distribution sense when $\omega \in L_{\textrm{loc}}^1$ as in the case of vortex patch (VPa). 

If for a solution of the incompressible Euler equation $\omega (t, x) = \ep \delta_{\bar x(t)} (x) $ is a point mass concentrated at $\bar x(t)$ for some $t$, the general principle that the vorticity is only transported following the fluid flow suggests that $\omega$ {\it might} remain as such a point vortex for all $t$. The question is,  as the singular vorticity generates a singularly rotational part 
\[
\frac \ep{|x-\bar x|^2} (\bar x_2 - x_2, x_1 - \bar x_1)^T
\]
of $\mbv$, following what vector field should the vortex point move? Since the above singular vector field is  purely rotational and does not move that particle at $\bar x(t)$ away, it is reasonable to expect that the dynamics of the vortex point is governed only by the remaining smooth part of $\mbv$  i.e. 
\[
\p_t \bar x = \big(\mbv - \frac \ep{|x-\bar x|^2} (\bar x_2 - x_2, x_1 - \bar x_1)^T\big)|_{x=\bar x}.
\]
This well known result is rigorously established by considering a family of solutions whose initial vorticity limiting (weakly) to a $\delta$-function
(cf. \cite[Theorem 4.1, 4.2]{marchioro1994book}).  For traveling waves where $\bar x(t) = ct \mbe_1$, this translates to  the following weaker form of \eqref{intro:vorticityeq}
\be \label{intro:ceq}
c \mbe_1 =\big( \mbv - \frac \ep{|x|^2} (-x_2, x_1)^T \big)|_{x=0}.
\ee
Indeed, our own analysis of the vortex patch confirms this for we found that the traveling speed $c$ for the travel vortex patch satisfies \eqref{intro:ceq} as the diameter of the patch converges to zero. 

As another indication how traveling waves with point vortex or vortex patch are different from those close to shear flows, it is worth pointing out that \eqref{intro:ceq} and \eqref{intro:vorticityeq} imply that the vortex point or some point in the vortex patch is a stagnation point.  That is, the physical horizontal speed coincides with the traveling speed $c$.  Usually stagnation only occurs as a limiting case in traveling waves constructed via bifurcation from shear flows.  They often coincide with the development of a singularity. 

Note that the water wave problem is invariant with respect  to reflection in $x_1$. For simplicity in this paper, we consider symmetric traveling wave solutions $(\Omega, \mbv)$ where $\Omega$ and $v_1$  are {\it even in $x_1$ and $v_2$ is odd in $x_1$}. 

Our main theorems are outlined in the statement below, while their more precise versions are given in Section \ref{setup:section}. 

\begin{main} 
\begin{enumerate}
\item For $0<|\ep| \ll 1$, there exists a unique  traveling water wave solution which is even in $x_1$ with small amplitude and small velocity whose vorticity is given by a delta mass of strength $\ep$ away from the surface.
\item This solution curve of traveling water waves with a point vortex can be extended globally with one of the possibilities: 
\begin{itemize} 
\item either $\eta$ or $\mbv$ becomes unbounded, 
\item the point vortex location becomes arbitrarily close to the water surface along the solution curve, or 
\item there exists a nontrivial irrotational traveling water wave, with gravity and surface tension, with an interior stagnation point. 
\end{itemize} 
\item For $0<|\ep| \ll 1$, there exist traveling water waves which is even in $x_1$ with small amplitude and small velocity whose vorticity of total strength $\ep$ is compactly supported in a small disk-like region away from the water's surface. 
\end{enumerate}  
\end{main} 

The construction of small traveling water waves with a point vortex, which is given in section \ref{local:pointvortexsection}, is based on a fairly simple implicit function theory argument.   

In section \ref{gobal:pointvortexsection}, we use a degree theoretic global bifurcation argument to extend the local bifurcation curve beyond the neighborhood of $0$ and into the finite-amplitude, speed, and vorticity regime.  Part (2) is the result of this process, an alternative theorem in the spirit of Rabinowitz \cite{rabinowitz1971some}. The global bifurcation curve must either be unbounded, with the separation between the point vortex and the surface limiting to $0$ along some sequence, or contain a nontrivial irrotational traveling water wave with gravity and surface tension, whose velocity at some point is equal to the wave speed $c$.  Traditionally with global bifurcation arguments, one begins with a statement of this type, and then uses \emph{a priori} estimates and nodal arguments to rule out one or more of the alternatives.  This was the procedure used by Amick, Fraenkel, and Toland to prove the famous  Stokes' conjecture (cf. \cite{amick1982stokes}), and it was the means by which Constantin and Strauss showed that the limiting waves of their global continuum must approach stagnation (cf. \cite{constantin2004exact}).  But, both of these results are for waves without surface tension, which is important because it renders the maximum principle dependent nodal arguments  tractable.  For rotational capillary waves, one of the authors has proved some theorems that go beyond simply the Rabinowitz-type (cf. \cite{walsh2009capillary}), but the large-amplitude regime is still mostly open.  Intuitively, though, point vortices are nearly irrotational so that one might hope that Theorem \ref{global:pointvortextheorem} can be extended to match the state-of-the-art for irrotational waves.  This is a very interesting question, but beyond the scope of the current paper. 

The third part of the main theorem, proved in Section \ref{localpatch:section}, deals with the other class of vorticity that we study, the \emph{vortex patch}.  These are solutions  where $\omega \in L^2(\Omega) \cap L^1(\Omega)$ and is supported in a compact region $D \subset\subset \Omega$.   One can view the point vortex as a limit of vortex patches when the size of the patch taken to zero.  We will require that  $\omega$  be continuous on $\Omega$, smooth in $D$, but not necessarily across $\partial D$.  Reformulating the vorticity equation in terms of the relative stream function, what results is a nonlinear elliptic free boundary problem, very much in the same vein as Constantin--Strauss.  However, rather than perturbing from a shear flow, our point of bifurcation will be radial solutions on a ball.  Moreover, we use conformal mappings to fix the support of the patch rather than streamline coordinates.  By construction, then, the streamlines on which the vorticity is nonzero will be closed.    On the other hand, the boundary motion of the air--water interface is dictated by \eqref{intro:bernoullieq} and \eqref{intro:kinematiceq}, just as in the irrotational setting.  To couple the two requires that some matching be done on $\partial D$.  

A few remarks about the hypotheses and possible extensions.  For the vortex patch problem, the interplay between the interior dynamics and the boundary motion are more intricate, which is why we need the optimal regularity furnished by fractional order Sobolev spaces.   Concerning the evenness assumption, while it turns out to simplify the computation, 
it should be noted that this type of symmetry is often expected for traveling waves.  Indeed, it is known that in a number of regimes, \emph{all} traveling waves (with monotonic profiles) have an axis of even symmetry (cf., e.g., \cite{craig1988symmetry,toland2000symmetry,constantin2007symmetry,toland2000symmetry,walsh2009symmetry}).  

A global continuation along the lines Theorem \ref{global:pointvortextheorem} seems an order of magnitude more difficult to carry out for the surface $\mathscr{S}_{\textrm{loc}}$.  The main obstruction is that, as can be seen in section \ref{gobal:pointvortexsection}, the global theory requires being able to prove compactness properties of the linearized operator not only at $0$, but anywhere along the continuum.   While this is not a problem for the free surface equations, the elliptic PDE for the vortex dynamics requires significantly more finesse.  We therefore do not consider this issue in the present work.  

The structure of the paper is as follows. In Section \ref{setup:section}, we start to reformulate the problems into forms more suitable for our analysis. In Section \ref{local:pointvortexsection}, we give the proof of local bifurcation for traveling waves with a point vortex.  The global theory of the point vortex case is developed using a degree theoretic argument in Section \ref{gobal:pointvortexsection}.  We address the vortex patch problem in Section \ref{localpatch:section}.  Additionally, some auxiliary lemmas and technical facts are collected in an appendix.   Finally, though we shall always define a notation when it is introduced, we include a table outlining our conventions at the end of the appendix.

\section{Framework and main theorems} \label{setup:section} 

In this paper, we seek traveling waves near the trivial state, namely $|\mbv| \ll 1$ and $|\eta| \ll 1$ where $\eta$ is given in the definition of $\Omega$ in \eqref{intro:travelingW}. In the interior of $\Omega$, the incompressible Euler Equation of traveling waves is equivalent to the vorticity equation \eqref{intro:vorticityeq} (possibly in a weak sense).  In addition, there are two more conditions on $\p \Omega$, the \emph{Bernoulli condition}, and the \emph{kinematic  condition}, which come from \eqref{intro:kinematic} and \eqref{intro:tension}, respectively. We will introduce the stream function and the velocity potential to reformaute the problem. \\

\noindent {\bf Stream functions.} 
On the one hand, in two space dimensions, it is convenient to write a divergence free vector field $\mathbf{v}$ as the skew gradient of a stream function $\Psi$ 
\[
\mathbf{v} = \nabla^\perp \Psi =: ( -\p_{x_2}  \psi, \p_{x_1} \psi)^T, \qquad \Delta \Psi = \omega.
\]
In the periodic case (Per), even though the domain $\Omega$ is homotopic to $\BBS^1$, and thus not simply connected, the decay of $\mbv$ at $x_2 = -\infty$ and the periodicity of $v_1$ ensure that $\Psi$ is single valued. Based on the convolution with the Newtonian potential, the vorticity $\omega$ naturally generates a corresponding part $\ep \CG(x)$ of the stream function where $\CG$ satisfies 
\[
\Delta \CG = \frac \omega\ep - \delta(\cdot - 2 \mbe_2)
\]
along with $2\pi L$-periodicity in $x_1$ in the periodic case (Per). More explicitly, in the localized case (Loc),
\be \label{setup:CG-loc}
\CG = (\frac 1{2\pi \ep } \log |\cdot|) * \omega - \frac 1{2\pi} \log |\cdot - 2 \mathbf{e}_2|. 
\ee
In the periodic case (Per) we need to sum up the above stream functions generated by the vorticity in each period and thus for $x_1\in [0, 2\pi L)$, 
\be \label{setup:CG-per}
\CG = \frac 1{2\pi} \text{P.V.} \sum_{k \in \mathbb{Z}} \big((\log |\cdot - 2k\pi L \mathbf{e}_1|) * \frac \omega\ep - \log |\cdot -  (2k \pi L, 2)^T|\big).
\ee
The `P.V.'  above means the principle value which ensures certain convergence of the summation. In the point vortex case (PtV) when $\omega (x)= \ep \delta(x)$, obviously the above convolutions simply yield a logarithm. While note $(2k \pi L, 2)^T \notin \Omega$ as $\eta \ll 1$ is assumed, the purpose of the second logarithm term in the summand above is to ensure that $\nabla \CG$ has better decay at $x=\infty$, namely $\nabla \CG = O(\frac 1{|x|^2})$ in all cases if $\omega$ is compactly supported.   Physically, it corresponds to the standard trick of introducing a phantom point vortex in the air region with equal but opposite vortex strength to correct for the lack of integrability of the Newtonian potential in $\mathbb{R}^2$. With $\CG$ given as above, there exists a harmonic function $\psi_\CH$ on $\Omega$ such that 
\be \label{setup:v-stream}
\mathbf{v} = \nabla^\perp \Psi = \nabla^\perp (\psi_\CH + \ep \CG). 
\ee
We will seek traveling wave solutions with $\psi_\CH \in \dot H^1(\Omega)$.\\

\noindent {\bf Velocity potentials.} 
On the other hand, an irrotational vector field can be (maybe locally) written as a gradient field. In the localized case (Loc) where $\Omega$ is simply connected, obviously any $\psi_\CH \in \dot H^1(\Omega)$ has a conjugate harmonic function $\varphi_\CH \in \dot H^1(\Omega)$ such that $\nabla^\perp \psi_\CH = \nabla \varphi_\CH$. In the periodic case (Per) where $x_1 \in \BBS^1$, the decay assumption on $\nabla \psi_\CH$ as $x_2 \to -\infty$ and the harmonicity of $\psi_\CH$ imply that the circulation $\oint \nabla^\perp \psi_\CH \cdot d \vec s$ vanishes along any closed curve. Therefore, $\psi_\CH$ also has a single valued conjugate harmonic function $\varphi_\CH \in \dot H^1 (\Omega)$. 

Since the vorticity $\omega$  of the traveling wave solutions $(\mathbf{v}, \eta)$ under our consideration is supported away from the free surface $\p \Omega$, $\nabla^\perp \CG$ is also irrotational in a neighborhood of $\p \Omega$. Therefore, there exists a function 
$\Theta$ (multi-valued in the periodic case (Per)) such that $\nabla \Theta = \nabla^\perp \CG$ and we have the decomposition  
\be 
\mathbf{v} = \nabla \Phi =:  \nabla \varphi_{\mathcal{H}} + \epsilon \nabla \Theta.\label{intro:decompv} 
\ee

\noindent {\bf Bernoulli equation in terms of steam functions.} 
The incompressible Euler equation of irrotational velocity fields and boundary condition \eqref{intro:tension} imply the \emph{Bernoulli equation} of traveling waves on $\p \Omega$
\[ 
-c \partial_{x_1}  \Phi + \frac{1}{2} |\nabla \Phi|^2 +  g x_2 + \alpha^2 \kappa = b \qquad \textrm{on } \{ x_2 = 1+ \eta(x_1)\}
\]
where $b$ is a constant. Though the velocity potential $\Phi$ might be multi-valued as pointed out  above,  note here only $\nabla \Phi$ is present which is single-valued. This fact allows us to write the Bernoulli equation in terms of the stream function only
\[
c (\p_{x_2} \psi_\CH + \ep \p_{x_2} \CG) + \frac 12 |\nabla \psi_\CH + \ep \nabla \CG|^2 +  g x_2 + \alpha^2 \kappa = b \qquad \textrm{on } \{ x_2 = 1+ \eta(x_1)\}.
\]
Let $\psi$ be the trace of $\psi_{\mathcal{H}}$ on the free surface 
i.e., 
\[ 
\psi(x_1) = \psi_{\mathcal{H}}(x_1, 1+ \eta(x_1))
\] 
and thus $\psi_{\mathcal{H}} := \mathcal{H}(\eta) \psi$ is the harmonic extension of $\psi$ to the fluid domain (cf. Lemma \ref{appendix:propGlemma}). From the Bernoulli condition, we arrive at the following equation of only the variable $x_1$:
\be \begin{split} 
0 &= c\left( \mathcal{G}(\eta) \psi + \epsilon( - \eta^\prime,1)^T \cdot \nabla \CG \right) + \frac{1}{2} \left(\mathcal{G} (\eta) \psi + \epsilon( -\eta^\prime, 1)^T \cdot \nabla \CG \right)^2 \\
& \qquad - \frac{1}{2(1+(\eta^\prime)^2)} \left( \psi' - \eta^\prime \mathcal{G} (\eta) \psi + \epsilon (1+(\eta^\prime)^2) \partial_{x_1} \CG \right)^2 
+ g (\eta + 1) + \alpha^2 \kappa(\eta) - b, 
\end{split} \label{intro:bernoullieq} \ee
where $\p \CG$ is evaluated at $(x_1, 1+ \eta(x_1))$ and 
\[ 
\mathcal{G}(\eta) := \sqrt{1+ (\eta^\prime)^2}\mathcal{N}(\eta),
\]
and $\mathcal{N}(\eta)$ is the Dirichlet-to-Neumann operator on $\Omega$; we recapitulate some of  the properties of these operators in Lemma \ref{appendix:propGlemma}. The constant $b$ can be determined explicitly. In the localized case (Loc), by taking $x_1 = \infty$, we obtain 
\be \label{setup:b-loc}
b=g.
\ee
In the periodic case (Per), \eqref{intro:eta} implies 
\[ \begin{split} 
b=b(\epsilon; \eta, \psi, c) =& g + \frac{1}{2\pi L} \int_{-\pi L}^{\pi L} \bigg[ c\left( \mathcal{G}(\eta) \psi + \epsilon( -\eta^\prime, 1)^T \cdot \nabla \CG \right) + \frac{1}{2} \left( \mathcal{G}(\eta) \psi + \epsilon( -\eta^\prime, 1)^T \cdot \nabla \CG \right)^2 \\
& - \frac{1}{2(1+(\eta^\prime)^2)} \left( \psi' - \eta^\prime \mathcal{G} (\eta) \psi + \epsilon (1+(\eta^\prime)^2) \partial_{x_1} \CG \right)^2 + \alpha^2 \kappa \bigg] \, dx_1. 
\end{split} \]
In the above integral, the $\kappa$ term integrates to $0$ due to the periodicity of $\eta$. Moreover, 
\[
\int_{-\pi L}^{\pi L} [\mathcal{G}(\eta) \psi + \ep ( -\eta^\prime, 1)^T \cdot \nabla \CG] \,  dx_1 = \int_{\p \Omega} N \cdot \nabla(\psi_\CH + \ep \CG) \,  ds =  \int_{\Omega} \omega \, dx = 1.
\]
Therefore we obtain in the periodic case (Per)
\be \label{localpoint:defbper} \begin{split} 
b=b(\epsilon; \eta, \psi, c) =& g + c\ep+ \frac{1}{2\pi L} \int_{-\pi L}^{\pi L} \bigg[ \frac{1}{2} \left( \mathcal{G}(\eta) \psi + \epsilon( -\eta^\prime, 1)^T \cdot \nabla \CG \right)^2 \\
& - \frac{1}{2(1+(\eta^\prime)^2)} \left( \psi' - \eta^\prime \mathcal{G} (\eta) \psi + \epsilon (1+(\eta^\prime)^2) \partial_{x_1} \CG \right)^2 \bigg] \, dx_1. 
\end{split} 
\ee

\noindent{\bf Kinematic equation.}
In terms of $\psi$ and $\CG$, the kinematic condition \eqref{intro:kinematic} for traveling waves is written as
\be 
0  = c \eta^\prime +  \psi' + \epsilon (1, \eta^\prime)^T \cdot \nabla \CG. \label{intro:kinematiceq} 
\ee
Equations \eqref{intro:bernoullieq} and \eqref{intro:kinematiceq} in terms of the velocity are derived, e.g., in \cite{sulem1999nlsbook} for the unsteady problem, here we have simply adapted them to the steady regime and transformed it using the stream function which is more convenient when the interior vorticity is treated. \\


\noindent {\bf Point vortex problem.} 
In the localized case (Loc), $\omega = \ep \delta(x)$ and thus \
\be \label{setup:CG-Loc1}
\CG (x) = \frac 1{2\pi} \log |x| - \frac 1{2\pi} \log|x-2\mbe_2|.
\ee
As discussed in Section \ref{intro:section},  \eqref{intro:vorticityeq} is replaced by \eqref{intro:ceq}. In the localized case (Loc) where $x_1 \in \BBR$, according to \eqref{setup:CG-loc} and the symmetry of $\mbv$ in $x_1$, we obtain 
\be \label{localpoint:ceq}
c = - (\partial_{x_2} \psi_{\mathcal{H}})(0) - \frac{\ep}{4\pi}.
\ee
In the periodic case (Per) where $x_1 \in L \BBS^1$, \eqref{setup:CG-per} implies instead 
\be \label{setup:CG-Per1}
\CG (x) = \frac 1{2\pi} \text{P.V.} \sum_{k \in \mathbb{Z}} \big(\log |x - 2k\pi L \mathbf{e}_1| - \log |x -  (2k \pi L, 2)^T|\big)
\ee
and 
\be \label{setup:ceq-per}
c = - (\partial_{x_2} \psi_{\mathcal{H}})(0) - \frac{\ep}{4\pi} \sum_{k=-\infty}^\infty \frac 1{k^2 \pi^2 L^2+1}.
\ee
In the point vortex case (PtV), we will seek solutions $(\eta, \psi, c)$ with $\eta$ even in $x_1$ and $\varphi$ odd in $x_1$, so that \eqref{intro:bernoullieq}, \eqref{intro:kinematiceq}, \eqref{setup:b-loc}, and \eqref{localpoint:ceq} in the localized case (Loc), or \eqref{intro:bernoullieq}, \eqref{intro:kinematiceq}, \eqref{localpoint:defbper}, and \eqref{setup:ceq-per} in the periodic case (Per), are satisfied. 

To state the main theorems in the point vortex case (PtV), we introduce the following spaces of even functions. For the localized case (Loc), let 
\[
H_e^k (\BBR) := \{ f \in H^k(\BBR) : f \text{ is even in } x_1\}, 
\]
$\dot H_e^k(\BBR)$ be the corresponding homogeneous space, and 
\be \label{intro:defspaceX}
X := H_e^k (\BBR) \times \big(\dot H_e^k (\BBR) \cap \dot H_e^{\frac 12} (\BBR)\big) \times \BBR.
\ee
For the periodic problem (Per), we simply replace $X$ with $X_{\textrm{per}}$, 
\be 
X_{\textrm{per}} := {H}_{\textrm{m}}^k(L \BBS^1) \times {H}_{\textrm{m}}^k(L \BBS^1) \times \mathbb{R}\label{intro:defspaceXper} \ee
where 
\[
{H}_{\textrm{m}}^k(L\BBS^1) := \{ f  \in H^k(L\BBS^1) :   f \textrm{ has mean } 0, \text{ even in } x_1\}.
\]
Our first theorem establishes the existence and uniqueness of a curve bifurcating from the trivial solutions. 

\begin{theorem}[Point vortex local bifurcation] \label{local:pointvortextheorem} Consider the traveling water wave problem with a point vortex at the origin \eqref{intro:bernoullieq}, \eqref{intro:kinematiceq}, \eqref{setup:b-loc}, and \eqref{localpoint:ceq} in the localized case \emph{(Loc)}, or \eqref{intro:bernoullieq}, \eqref{intro:kinematiceq}, \eqref{localpoint:defbper}, and \eqref{setup:ceq-per} in the periodic case \emph{(Per)}. 
The following statements hold.  
 \begin{itemize}
\item[(i)]  There exists $\epsilon_0 > 0$ and a $C^\infty$-curve of solutions 
\[ 
\mathscr{C}_{\mathrm{loc}}= \{ (\epsilon, \eta(\epsilon), \psi(\epsilon), c(\epsilon)) : | \epsilon | < \epsilon_0 \} \subset \mathbb{R} \times X ~(\textrm{or } \mathbb{R} \times X_{\mathrm{per}}) 
\]
for any $k>\frac 32$ with 
\[ 
(0, \eta(0), \psi(0), c(0)) = (0,0,0,0).
\]
Moreover, in a sufficiently small neighborhood of $0$ in $\mathbb{R} \times X$ (or $\mathbb{R} \times X_{\mathrm{per}}$), $\mathscr{C}_{\mathrm{loc}}$ comprises all solutions.
\item[(ii)] In the localized case \emph{(Loc)}, the solutions have the asymptotic form 
\[
c(\epsilon) = - \frac \epsilon {4\pi} + o(\epsilon^2), \qquad |\psi (\ep)|_{H^k} = O(\ep^3)
\]
and 
\[ 
|\eta(\epsilon) -  \frac {\ep^2}{4\pi^2}  (g -\alpha^2 \p_{x_1}^2)^{-1} \big( \frac {x_1^2 -1}{(1+x_1^2)^2} \big) |_{H^k} = O(\ep^3).
\]
\item [(iii)] In the periodic case \emph{(Per)}, the solutions have the asymptotic form 
\[ 
c(\epsilon) = \ep \tilde c_0 + o(\epsilon^2), \qquad 
|\psi (\ep)|_{H^k} = O(\ep^3), \qquad |\eta - \ep^2 \eta_*|_{H^k} = O(\ep^3) 
\]
where
\[ \begin{split} 
\tilde c_0=& -\frac 1{4\pi}  \sum_{k=-\infty}^\infty \frac 1{k^2 \pi^2 L^2+1}, \\
\eta_* = &- (g -\alpha^2 \p_{x_1}^2)^{-1} \big(\tilde c_0 \p_{x_2} \CG + \frac 12 (\p_{x_2} \CG)^2  - \frac 1{2\pi L} \int_{-\pi L}^{\pi L} \tilde c_0 \p_{x_2} \CG + \frac 12 (\p_{x_2} \CG)^2 dx_1 \big),
\end{split} \]
$\CG$ is defined as in \eqref{setup:CG-Per1}, and $\nabla \CG$ is evaluated at $x_2 =1$. 
\end{itemize}	
\end{theorem} 

\begin{remark}
(a) Note the above existence and uniqueness results hold for all $k > \frac 32$, so the solution $(\eta, \psi)$ are actually $C^\infty$ functions. \\
(b) The above local uniqueness is stated in the framework of $(\eta, \psi, c) \in X$ (or $X_{per}$). However, if a velocity field $\mbv$ satisfies $\nabla \cdot \mbv =0$, $\nabla \times \mbv = \ep \delta$, and $\mbv \in L^2$ outside any neighborhood of $0$, it can be written in the form given in \eqref{setup:v-stream}. Therefore, the local uniqueness holds in the space of such vector fields. 
\end{remark} 

The proof, which is given in section \ref{local:pointvortexsection}, is based on a fairly simple implicit function theory argument. In section \ref{gobal:pointvortexsection}, for the periodic case (Per), we use a degree theoretic global bifurcation argument to extend the curve beyond the neighborhood of $0$ and into the finite-amplitude, speed, and vorticity regime.  The result is the following alternative theorem in the spirit of Rabinowitz \cite{rabinowitz1971some}.

\begin{theorem}[Point vortex global bifurcation] \label{global:pointvortextheorem} For $k\ge 3$ in the definition of $X_{\mathrm{per}}$, there exists a connected set $\mathscr{C} \subset \mathbb{R} \times X_{\mathrm{per}}$ of solutions to \eqref{intro:bernoullieq}, \eqref{intro:kinematiceq}, \eqref{localpoint:defbper}, and \eqref{setup:ceq-per} with $\mathscr{C}_{\mathrm{loc}} \subset \mathscr{C}$.  One of the following alternatives must hold:
\begin{itemize}
\item[(i)] There is a sequence $\{(\epsilon_n, \eta_n, \psi_n, c_n)\} \subset \mathscr{C}$ that is unbounded in $\mathbb{R} \times X_{\mathrm{per}}$; 
\item[(ii)] there exists a nontrivial irrotational (i.e. $\ep=0$) traveling wave solution with stagnation point at the original, or
\item[(iii)] along some sequence in $\{(\epsilon_n, \eta_n, \psi_n, c_n)\} \subset \mathscr{C}$, we have $\eta_n(0) \to -1$. 
\end{itemize}  
\end{theorem}

\noindent {\bf Localized waves with vortex patch case.} 
In this case, we will look for solutions $(\eta, \psi, \omega, c)$ of \eqref{intro:bernoullieq}, \eqref{intro:kinematiceq},  \eqref{setup:b-loc}, 
and \eqref{intro:vorticityeq}. In fact, \eqref{intro:vorticityeq} is essentially only required on $D := \supp(\omega)$ while it should be satisfied on $\Omega$ in the distributional sense. Since $D\subset \subset \Omega$ is an unknown, through a more elaborated procedure using conformal mappings in Section \ref{localpatch:section}, \eqref{intro:vorticityeq} will be further transformed to a form more suitable for our analysis. Using a (local) bifurcation theory argument, we are able to construct a curve of small-amplitude, small vorticity, slow speed, and small patch solutions in the neighborhood of the trivial solution. 

\begin{theorem}[Vortex patch local bifurcation] \label{localpatch:bifurcationtheorem}
For  any $s > 3/2$ and integer $k_0\ge1$ or $k_0=\infty$, there exists a three-dimensional $C^{k_0}$ surface of solutions to the vortex patch problem
\[ 
\mathscr{S}_{\mathrm{loc}} = \{ (\eta(\delta, \epsilon, \tau), \psi(\delta, \epsilon, \tau), \omega(\delta, \epsilon, \tau), c(\delta, \epsilon, \tau)) : (\epsilon, \delta, \tau) \in \mathcal{U} \}, \]
where $\mathcal{U} := [0,\epsilon_0) \times (0,\delta_0) \times [0, \tau_0)$ for some $\epsilon_0, \delta_0, \tau_0 > 0$, 
\[ 
\mathscr{S}_{\mathrm{loc}} \subset H^{s+1}_{\mathrm{e}}(\mathbb{R}) \times \big(\dot H_e^s(\mathbb{R}) \cap \dot H_e^{\frac 12} (\BBR)\big) \times H_{\mathrm{e}}^1 (\BBR^2)\times \mathbb{R} .\]
The parameterization is such that 
\[ (\eta(0, 0, 0), \varphi(0,0, 0), \omega(0, 0, 0), c(0, 0, 0)) = (0,0,0,0), \qquad c = \ep\big( -\frac 1{4\pi} + O(\ep+\delta)\big).
\]
Moreover, $\omega \in H^{s+\frac 12} (D)$ where $D:=\supp(\omega) \subset \Omega$ and $\p D$ is an $H^s$ closed curve with asymptotic form: 
\[ 
\partial D(\epsilon, \delta, \tau) = \{ \delta ( \cos{\theta} + \tau \sin{(2\theta)}, \sin{\theta} - \tau\cos{(2\theta)}) + O\big(\delta^2(\delta + \ep)\big) : \theta \in [0,2\pi] \}. 
\]
\end{theorem}

\begin{remark} 
1.) Here we see that, for fixed $\ep>0$, as $\delta \to 0+$, the wave speed $c$ converges to the wave speed of that of the point vortex problem. \\
2.) In fact, such a smooth family of solution is found for any fixed {\it vorticity strength function} $\gamma$ from a large class of functions to be introduced in Section \ref{localpatch:section}. By choosing different $\gamma$, potentially a large family of such traveling waves can be found which differ in the $O\big( \delta^2(\delta +\ep)\big)$ term in the above. 
\end{remark}

\section{Small amplitude waves with a point vortex} \label{local:pointvortexsection}

Let us first turn our attention to the point vortex problem.  We will only focus on the localized case (Loc) and the proof for the periodic case (Per) follows from exactly the same procedure. 
Since the traveling wave solutions we are seeking are bifurcated from the trivial solution, $\eta$ and $\psi$ are of order $O(\ep)$. In principle, the wave speed $c$ can be anything for the trivial solution, equation \eqref{setup:b-loc} implies that $c$ is also of this order, which is one of the major differences between the problems with localized vorticity and the one near the shear flows. Therefore it is more convenient to rescale 
\[
\eta = \ep \tilde \eta, \quad \psi = \ep \tilde \psi, \quad c= \ep \tilde c. 
\]
Using  abstract formalism, we can express the governing equations \eqref{intro:bernoullieq}, \eqref{intro:kinematiceq}, \eqref{setup:b-loc}, and \eqref{localpoint:ceq}  as
\[
\mathcal{F}(\epsilon; \tilde \eta, \tilde \psi, \tilde c) = 0,
\]
where $\mathcal{F} = (\mathcal{F}_1, \mathcal{F}_2, \mathcal{F}_3) : \mathbb{R} \times X \to Y,$ and 
\be \label{localpoint:defF} \begin{split}
\mathcal{F}_1
& := \tilde c\ep \left( \mathcal{G}(\ep \tilde \eta) \tilde \psi + ( - \ep \tilde \eta^\prime,1)^T \cdot \nabla \CG \right) + \frac{\ep}{2} \left(\mathcal{G} (\ep \tilde \eta) \tilde \psi + ( -\ep \tilde \eta^\prime, 1)^T \cdot \nabla \CG \right)^2 \\
& \qquad - \frac{\ep}{2(1+(\ep \tilde \eta^\prime)^2)} \left(\tilde \psi' - \ep \tilde \eta^\prime \mathcal{G} (\ep \tilde \eta) \tilde \psi + (1+(\ep \tilde \eta^\prime)^2) \partial_{x_1} \CG \right)^2 
+ g \tilde \eta  + \frac {\alpha^2}\ep  \kappa(\ep \tilde \eta), \\
\mathcal{F}_2
& := \tilde c \ep \tilde \eta^\prime + \tilde \psi' + (1, \ep \tilde \eta^\prime)^T \cdot \nabla \CG\\
\mathcal{F}_3
& := \tilde c + (\partial_{x_2} \tilde \psi_{\mathcal{H}})(0) + \frac{1}{4\pi} 
\end{split} \ee
where $\CG$ is defined in \eqref{setup:CG-Loc1}, $\nabla \CG$ is evaluated at $x_2 = 1+\ep \tilde \eta$, $b$ is replaced by $g$ as in \eqref{setup:b-loc}, the space $X$ is defined as in \eqref{intro:defspaceX};
and $Y$ 
is taken to be
\[ 
Y := H_e^{k-2}(\mathbb{R}) \times \big(\dot H_e^{k-1} (\BBR) \cap \dot H_e^{-\frac 12}(\mathbb{R}) \big) \times \mathbb{R}
. \]
Obviously small traveling waves with a point vortex at the origin corresponds to a zero point of $\mathcal{F}$. We are now ready to prove the local bifurcation theorem for point vortices.   
 
\begin{proof}[Proof of Theorem \ref{local:pointvortextheorem}]  Recall the $\nabla \CG$ term has the explicit form 
\[ 
\nabla \CG = \frac 1{2\pi} \big(\frac x{|x|^2} - \frac {x-2\mbe_2}{|x-2\mbe_2|^2}\big)
|_{\{x_2 = 1+ \ep \tilde \eta(x_1)\}}. 
\]
Thus the $\nabla \CG$ term depends analytically on $(\tilde \eta, \tilde \psi, \tilde c)$ for $|\ep| \ll 1$.  Similarly, the operators $\mathcal{H}$, $\mathcal{G}$, and $\kappa$ are also smooth (cf. Lemma \ref{appendix:propGlemma}), implying that $\mathcal{F}$ is of $C^\infty$ class  (in fact they are analytic, but this is beyond our needs.)  

For $\ep =0$, is clear that the following point is in the zero-set of $\mathcal{F}$: 
\[
\tilde \eta_0= 0, \quad \tilde \psi_0= - \CG|_{x_2=1} = 0, \quad \tilde c_0=-\frac 1{4\pi}.
\]
Now, denoting
\[ 
\big(D \mathcal{F}\big)_0 := (D_{\tilde \eta}, D_{\tilde \psi}, D_{\tilde c}) \mathcal{F}(0;0,0,-\frac 1{4\pi}), 
\]
a simple computation reveals that
\[ 
\big(D \mathcal{F}\big)_0 = \left( \begin{array}{ccc} g - \alpha^2 \partial_{x_1}^2 & 0 & 0 \\
0 & \p_{x_1} & 0 \\
0 & (\partial_{x_2} \langle  \mathcal{H}(0),\cdot\rangle)|_{(0,0)} & 1  \end{array} \right) \in \mathcal{L}(X; Y) 
.\]
Each diagonal entry here is invertible (see Lemma \ref{appendix:propGlemma}), and thus $D\mathcal{F}(0)$ is an isomorphism.  The implicit function theorem immediately implies the local existence and uniqueness of a curve of zero points of $\mathcal{F}$ parametrized by $|\ep| \ll 1$. Direct expansion shows $\tilde c = -\frac 1{4\pi} +O(\ep)$, $\tilde \psi= O(\ep^2)$ and 
\[\begin{split}
\tilde \eta =& \ep  (g -\alpha^2 \p_{x_1}^2)^{-1} \big( - \tilde c_0 \p_{x_2} \CG - \frac 12 (\p_{x_2} \CG)^2 + \frac 12 (\p_{x_1}\CG)^2 \big) + O(\ep^2) \\
=& \frac \ep{4\pi^2}  (g -\alpha^2 \p_{x_1}^2)^{-1} \big( \frac {x_1^2 -1}{(1+x_1^2)^2} \big)  + O(\ep^2)
\end{split}\]  
where $\nabla \CG$ was evaluated at $x_2 =1$. This in turn yields $\tilde c = -\frac 1{4\pi} + O(\ep^2)$. The proof for the localized case (Loc) is complete.  
\end{proof}

\section{Small amplitude waves with a vortex patch} \label{localpatch:section}

We now consider traveling waves in the localized case (Loc) where $x_1 \in \BBR$ and the vorticity $\omega$ is supported in a compact region $D:= \supp(\omega) \subset \subset \Omega$. As outlined in section \ref{setup:section}, 
this problem is equivalent to the system given by \eqref{intro:bernoullieq}, \eqref{intro:kinematiceq},  \eqref{setup:b-loc}, and \eqref{intro:vorticityeq} with unknowns  $(\eta, \psi, \omega, c)$ with $\epsilon := \int_\Omega \omega \, dx$ as a small parameter. Since $\omega$ is supported on an unknown domain $D$, we first further rewrite the problem to address this difficulty.

\subsection{Reformulation} \label{localpatch:streamfunctionsectionn}

As one of the unknowns, we consider the vortical region $D$ as being a perturbation of $B_{\delta}$, the ball of radius $\delta$ centered at the origin, where $\delta$ is fixed and small. As we assumed $\p \Omega = \{ x_2 = \eta(x_1) + 1 \}$ and $|\eta|<<1$, there is a positive separation between the patch $D$ and the air--water interface.

As given in \eqref{setup:v-stream}, recall $\Psi=\psi_\CH + \CG$ is the stream function of $v$ where $\CG$ is given in \eqref{setup:CG-loc}. Define the relative stream function 
\[
f:= \Psi + c x_2,
\] 
then the vorticity equation \eqref{intro:vorticityeq} takes the equivalent form
\be 
\nabla f \cdot \nabla \Delta f = \nabla^\perp f \cdot \nabla \omega = 0.\label{localpatch:poissonbracket} 
\ee
Before proceeding further, it will  be convenient to scale the physical variables, the relative stream function, and the vorticity.  With that in mind,  define $D_0 := \frac{1}{\delta} D$ and  $\widetilde{f}, \widetilde \omega: D_0 \to \mathbb{R}$ by\footnote{Here and elsewhere in this section, we adopt the convention that quantities scaled in $\delta$ and $\epsilon$ are identified with tildes.}
 \be 
f(x)  =: \epsilon \widetilde{f}(\frac{x}{\delta}), \qquad \omega(x) = \frac{\epsilon}{\delta^2} \widetilde{\omega}(\frac{x}{\delta}). \label{localpatch:defftilde} 
\ee
Note $D_0$ is thus a small perturbation of $B_1$ and that, as $\Delta f = \omega$,  
\[ \Delta \widetilde{f} = \widetilde{\omega} \textrm{ in } D_0, \qquad  \int_{D_0} \Delta\tilde{f} \, dx = \int_{D_0} \widetilde \omega \,  dx = 1.\]
One way to ensure that the vorticity equation \eqref{localpatch:poissonbracket} is satisfied is to require that the vortex lines and streamline coincide, i.e.  
\[ \Delta \widetilde{f} = \widetilde{\omega} = \gamma(\widetilde{f}) \qquad \textrm{ in } D_0,\]
for  some function $\gamma$ called the \emph{vorticity strength function}.     As $D_0 = \supp{(\widetilde{\omega})}$,  $\partial D_0$ is a vortex line ($\widetilde{\omega}|_{\partial D_0} = 0$) and thus a streamline where $\widetilde f$ is a constant. Without loss we may assume that $\widetilde f|_{\partial D_0} = 0$ so that, in total,  the scaled relative stream function $\widetilde{f} \in {H}_0^1(D_0)$ is the solution to the elliptic PDE 
 \be \left\{ \begin{array}{ll} \Delta \widetilde{f} = \gamma(\widetilde{f}) & \textrm{in } D_0 \\
\widetilde{f} = 0 & \textrm{on } \partial D_0.\end{array} \right.\label{localpatch:fsemilineareq} \ee

Since the above elliptic boundary value problem can be approximated by the one with $D_0 = B_1$, we impose some conditions on $\gamma$ so that it has a non-degenerate solution when $D_0 = B_1$. Namely, 
we assume
\begin{subequations}  \label{localpatch:gammaassumptions} 
\be \gamma \in C^N (\mathbb{R}), \qquad \gamma(0) = 0, \qquad \gamma^\prime(0) < 0, \qquad \gamma > 0 \textrm{ on } \mathbb{R}^-, 
\ee 
\be 
\textrm{there exists a negative radial solution $\widetilde{f}^*$ to \eqref{localpatch:fsemilineareq} with $D_0 = B_1$,} 
\ee
and
\be \Delta - \gamma^\prime(\widetilde{f}^*) \textrm{ is non-degenerate},
\ee \end{subequations}
where the integer $N> s+k_0+\frac 12$ with $s$ and $k_0$ being given in Theorem \ref{localpatch:bifurcationtheorem}. We will look for traveling wave solutions with $D_0$ close to $B_1$ for each such fixed $\gamma$. 

Solving \eqref{localpatch:fsemilineareq} will allow us to determine the vorticity, but before we can do that, we must represent the domain $D_0$ in an analytic way, which is accomplished by using near identity conformal mappings between $B_1$ and $D_0$. 
Let $\Gamma$ be a conformal mapping with domain the unit ball $B_1 \subset \mathbb{C}$, and satisfying
\begin{subequations} \label{localpatch:Gammaassumptions}
\be \partial_{\bar{z}} \Gamma = 0 \textrm{ in } B_1, \qquad \Gamma(0) = 0, \qquad \Gamma^\prime(0) = 1.  \label{localpatch:defGamma} \ee
By identifying $z = x_1 + i x_2 \in \mathbb{C}$ with the point $(x_1, x_2) \in \mathbb{R}^2$, we may  view the dilated domain $D_0$ as the image of $B_1$ under $\Gamma$, $D_0 := \Gamma(B_1)$, and the unscaled domain $D := \delta D_0$. As we are interested in vortex patches that are perturbations of $B_\delta$, we think of $|\Gamma(z) - z|$ as being small throughout $D_0$.   

Let us now briefly motivate \eqref{localpatch:defGamma}.  The first statement is just the conformality, while the second fixes the origin.  The third is made in order to eliminate a certain redundancy. Observe that for each $\sigma > 0$, $(\delta, \Gamma)$ and $(\delta/\sigma, \sigma \Gamma)$ each result in the same patch $D$.  By fixing $\Gamma^\prime(0) = 1$, we exclude all but $\sigma = 1$.  

In addition, since we look for traveling waves with the fluid domain and the stream functions even in $x_1$, 
we require that $D$ is symmetric over the $x_1$-axis, which, stated in terms of $\Gamma$, is equivalent to
\be \realpart{\Gamma} \textrm{ is odd in } x_1, \qquad \textrm{and} \qquad \imagpart{\Gamma} \textrm{ is even in } x_1.  \label{localpatch:symmetryassumptions} \ee \end{subequations}
It is clear that there is a one-to-one correspondence between a symmetric domain $D_0$ close to $B_1$ and such a near identity conformal mapping $\Gamma$ satisfying \eqref{localpatch:defGamma} and \eqref{localpatch:symmetryassumptions}.

The connection between \eqref{localpatch:symmetryassumptions} and the symmetry of the domain can be seen as follows.  As $\Gamma$ is analytic, we may express it as a power series
 \be \Gamma(z) = \sum_{n=1}^\infty a_n z^n = z + \sum_{n\geq 2} a_n z^n.\label{localpatch:Gammaexpansion1} \ee 
That $a_0 = 0$ and $a_1 = 1$ follows from \eqref{localpatch:defGamma}. Consider the term $a_n z^{n}$, for some $n \geq 1$.  Denote  $\theta_n := \arg{a_n}$ and $\theta := \arg{z}$, which implies 
\[ \arg{(a_n z^n)} = n\theta + \theta_n.  \] 
 Symmetry of $D$ over the $x_1$-axis translates to the requirement that 
 \begin{align*} \cos{\left( n(\pi - \theta) + \theta_n\right)} &= -\cos{(n\theta + \theta_n)}, \\
  \sin{\left(n(\pi -\theta) + \theta_n\right)} &= \sin{(n\theta + \theta_n)}.\end{align*}
 Expanding the first of these identities yields 
 \[ (-1)^{n+1} \cos{(n\theta - \theta_n)} = \cos{(n\theta + \theta_n)}. \]
 From this it is apparent that $\cos{\theta_n} = 0$ for $n$ even, $\cos{\theta_n} = \pm 1$, for $n$ odd.  That is
 \be  \realpart{a_{2n-1}} = a_{2n-1}, \qquad \realpart{a_{2n}} = 0, \qquad \textrm{for } n \geq 1. \label{localpatch:symmetry} \ee
As a consequence, we have \eqref{localpatch:symmetryassumptions}. Moreover, the above property of $a_n$ implies that $\Gamma$ takes the form 
\be \label{localpatch:beta}
\Gamma(z) = z+ \sum_{n=2}^\infty i^{n-1} \beta_n z^n, \qquad \beta_n \in \BBR.
\ee

Note that the correspondence between a conformal mapping $\Gamma$ defined on $B_1$ and the real part of its trace on $\BBS^1$ is one-to-one and onto. Let 
 \be 
 \beta = \beta(\theta) = \realpart{[\Gamma(e^{i\theta})-e^{i\theta}]}= \sum_{n=2}^\infty \beta_n \cos{(\frac{1}{2}(n-1)\pi + n \theta)}, \qquad \theta \in \BBS^1.
 \label{localpatch:defbeta} \ee
the trivial solution corresponds to $\beta \equiv 0$.  For each $s \geq 0$, define the space
 \be X^s := \{ \beta \in H^s(\BBS^1) : \beta = \sum_{n=2}^\infty \beta_n \cos{(\frac{1}{2}(n-1)\pi + n \theta)}, ~ \{ \beta_n\}_{n=2}^\infty \subset \mathbb{R} \}.\label{localpatch:defXs} \ee
 Obviously $\beta \in H^s$ if and only if $\sum_{n=2}^\infty n^{2s} \beta_n^2 < \infty$. 
  
We consider $\beta \in H^s(S^1)$ to be the unknown describing the shape of the vortex patch, and write $\Gamma = \Gamma(\beta)$ to emphasize that $\Gamma$ is defined in terms of $\beta$.  It is not hard to show that $\Gamma(\beta)$ is smooth and bounded in the appropriate spaces, cf. Lemma \ref{appendix:regGammaFtilde}.  Where there is no risk of confusion, we shall abuse notation and suppress this dependence entirely.   

Letting $\widetilde{F} := \widetilde{f} \circ \Gamma(\beta)$, the semi-linear problem for $\widetilde{f}$ becomes
\[  \left\{ \begin{array}{ll} \Delta \widetilde{F} = | \partial_z \Gamma|^2 \gamma(\widetilde{F}) & \textrm{on } B_{1} \\
\widetilde{F} = 0 & \textrm{on } S^1.  \end{array} \right. \]
We have thus managed to fix the domain.  Note, however, that the above equation will not suffice as a definition of $\widetilde{F}$, because we need in addition that $\int_{B_1} \Delta \widetilde{F}\, dx= \int_{D_0} \Delta \widetilde f dx = 1$.  With that in mind, we instead consider the problem of $(\widetilde F, a)$ 
\be 
\left\{ \begin{array}{ll} \Delta \widetilde{F} =a | \partial_z \Gamma|^2 \gamma(\displaystyle\frac{1}{a}\widetilde{F}) & \textrm{on } B_{1}\\
\widetilde{F} = 0 & \textrm{on } S^1 \\
\int_{B_1} \Delta\widetilde{F} \, dx = 1, ~ a> 0. & \end{array} \right. 
 \label{localpatch:Fsemilineareq} \ee
 
Now we proceed to transform  the system given by \eqref{intro:bernoullieq}, \eqref{intro:kinematiceq},  \eqref{setup:b-loc}, and \eqref{intro:vorticityeq} with unknowns  $(\eta, \psi, \omega, c)$ to a system of equations with unknowns $(\eta, \psi, \beta, c)$. The key here is \eqref{intro:vorticityeq} which is already expressed by the above semilinear elliptic problem.  

Given $\beta \in X^s$, by the nondegenercy assumptions on $\gamma$ in \eqref{localpatch:gammaassumptions}, for $\beta \in X^s$ in a neighborhood of $0$, there exists a locally unique solution $(\widetilde{F}, a)$ to \eqref{localpatch:Fsemilineareq}.  Indeed, we show in Lemma \ref{appendix:regGammaFtilde} that the solution will depend smoothly on $\beta$.   Once $\Gamma$ and $\widetilde{F}$ are known, the scaled vorticity support $D_0=\Gamma(B_1)$ and define $\widetilde \omega$ through 
\be \label{localpatch:tomega}
\widetilde{\omega} \circ \Gamma = |\partial_z \Gamma|^{-2} \Delta \widetilde{F} \text{ in } D_0  \quad \text{ and } \quad \widetilde\omega =0 \text{ in } (D_0)^c
\ee
which  is {\it expected} to be the scaled vorticity. 
Accordingly, we define the mapping
\be
\widetilde{H}: \beta \in X^s \mapsto \frac{1}{2\pi} a\int_{B_1} \log{|\Gamma(\cdot) - \Gamma(z)|} \gamma(\frac{1}{a} \widetilde{F}(z)) |\Gamma^\prime(z)|^2 \, dz 
\label{localpatch:defPsitilde} \ee
where $\widetilde{F}$ is the unique solution to \eqref{localpatch:Fsemilineareq}.  The smoothness of $\widetilde{H}$ is established in Lemma \ref{appendix:regPsi}.
 
Chasing definitions, $\widetilde{H} (\beta)$  is {\it expected} to be the scaled trace on $\partial D_0$ of the rotational part of the relative stream function corresponding to the vortex patch. We will need to choose $\Gamma$ appropriately so that we can take the solution $\widetilde F$ to the semilinear problem and find a corresponding relative stream function defined globally in $\Omega$ for the traveling wave solution we are seeking, for this, understanding the dependence of $\widetilde{H}$ on $\beta$ is critical. In fact, given $(\eta, \psi, \beta, c)$,  define 
\be \label{localpatch:tf} \begin{split}
\widetilde{f}(x) &= \frac{1}{\epsilon} \Psi(\delta x) + \frac{c}{\epsilon} \delta x_2
\\
& = \widetilde{\psi}_\CH(\delta x) + \frac{1}{2\pi}  \left(\log{|\cdot|} * \widetilde{\omega}\right)(x) - \frac{1}{2\pi} \log{|\delta x - 2\mathbf{e}_2|} + \delta \widetilde{c} x_2 + \mu 
\end{split} \ee
where 
\[ 
\widetilde{\psi}_\CH := \frac{1}{\epsilon} \psi_\CH, \qquad \widetilde{c} := \frac{1}{\epsilon} c,
\]
and $\mu$ is chosen so that the mean of $\widetilde f \circ \Gamma$ on $\BBS^1$ is zero. $\widetilde f$ is expected to be the relative stream function. Based on these definitions, we have 
\[
\Delta \widetilde f = \widetilde \omega, \qquad \widetilde H(\beta) = \frac 1{2\pi} (\log |\cdot| * \widetilde \omega)\circ \Gamma 
\]
and thus
\[
\widetilde f (x) =  \widetilde{\psi}_\CH(\delta x) +(\widetilde H(\beta) \circ \Gamma^{-1}) (x) - \frac{1}{2\pi} \log{|\delta x - 2\mathbf{e}_2|} + \delta \widetilde{c} x_2 + \mu.
\]

If $\widetilde f \circ \Gamma$ coincide with $\widetilde F$, which solves the elliptic problem \eqref{localpatch:Fsemilineareq}, then clearly $\widetilde f$ satisfies \eqref{localpatch:fsemilineareq} and thus it yields a velocity field $\mbv$ satisfying \eqref{intro:vorticityeq}. Since $\Delta \widetilde F = \Delta (\widetilde f \circ \Gamma)$, $\widetilde F|_{\BBS^1}=0$, and $\int_{\BBS^1} \widetilde f \circ \Gamma \, ds =0$, we only need 
\begin{align*} 
0 &=  \partial_\theta \left[\widetilde \psi_{\mathcal{H}} \circ \big(\delta \Gamma(\beta)\big) + \widetilde c\delta \imagpart{\Gamma(\beta)} - \frac{1}{2\pi} \log{|\delta \Gamma(\beta) - 2 \mathbf{e}_2|} + \widetilde{H}(\beta)\right]|_{S^1}.
\end{align*}
Note that the bracketed function, when considered on $\BBS^1$, is even in $x_1$ and thus the right side of the above equation is odd in $x_1$. This means that it enjoys the same symmetry as $\beta$ does, but may include the $n=1$ mode. Define 
\be 
Y^s := \{ \beta \in H^s(S^1) : \beta = \sum_{n=1}^\infty \beta_n \cos{(\frac{1}{2}(n-1)\pi + n \theta)}, ~ \{ \beta_n\}_{n=1}^\infty \subset \mathbb{R} \}.\label{localpatch:defYs}
\ee

Let us now put this in the framework of the implicit function theorem. As in the point point vortex case, we consider the scaled unknowns 
\[
\widetilde \eta= \frac 1\ep \eta, \qquad \widetilde \psi =\frac 1\ep \psi, \qquad \widetilde c = \frac c\ep 
\]
as well as $\beta$ which gives $\Gamma$, $D_0 = \Gamma(B_1)$, $\widetilde F$, $\widetilde \omega$. Based on the above discussion along with the Bernoulli equation \eqref{intro:bernoullieq} and the kinematic boundary condition \eqref{intro:kinematiceq},  the problem of traveling water waves with a vortex patch is transformed to a system in the form of 
\[
\mathcal{F} (\epsilon, \delta; \widetilde \eta, \widetilde \psi, \beta, \widetilde c) = 0,
\]
where 
\be \label{localpatch:defF} \begin{split}
\mathcal{F}_1
& := \widetilde c\ep \left( \mathcal{G}(\ep \widetilde \eta) \widetilde \psi + ( - \ep \widetilde \eta^\prime,1)^T \cdot \nabla \CG \right) + \frac{\ep}{2} \left(\mathcal{G} (\ep \widetilde \eta) \widetilde \psi + ( -\ep \widetilde \eta^\prime, 1)^T \cdot \nabla \CG \right)^2 \\
& \qquad - \frac{\ep}{2(1+(\ep \widetilde \eta^\prime)^2)} \left(\widetilde \psi' - \ep \widetilde \eta^\prime \mathcal{G} (\ep \widetilde \eta) \widetilde \psi + (1+(\ep \widetilde \eta^\prime)^2) \partial_{x_1} \CG \right)^2 + g \widetilde \eta  + \frac {\alpha^2}\ep  \kappa(\ep \widetilde \eta), \\
\mathcal{F}_2
& := \widetilde c \ep \widetilde \eta^\prime + \widetilde \psi' + (1, \ep \widetilde \eta^\prime)^T \cdot \nabla \CG\\
\mathcal{F}_3
& :=  \partial_\theta \left[\widetilde \psi_{\mathcal{H}} \circ \big(\delta \Gamma(\beta)\big) + \widetilde c\delta \imagpart{\Gamma(\beta)} - \frac{1}{2\pi} \log{|\delta \Gamma(\beta) - 2 \mathbf{e}_2|} + \widetilde{H}(\beta)\right]|_{S^1}.
\end{split} \ee
Here $\nabla \CG$ is evaluated at $x_2 = 1+ \ep \widetilde \eta (x_1)$, $\CG$ is defined as in \eqref{setup:CG-loc} with $\omega$ given by \eqref{localpatch:defftilde}, \eqref{localpatch:tomega}, and \eqref{localpatch:Fsemilineareq}, which are solely determined by $\beta$. Any zero point of $\mathcal{F}$ defines $\widetilde f$ by \eqref{localpatch:tf} which after rescaling \eqref{localpatch:defftilde} yields the velocity field $\mbv = \nabla^\perp (f - c x_2) \in L^2(\Omega)$, where $\Omega$ is given by $\ep \widetilde \eta$, for the traveling wave solution. To find zero points of $\CF$, we will prove in the appendix (cf. Lemma \ref{appendix:regF}.) that $\widetilde H$ is a $C^{k_0+1}$ mapping from $X^s$ to $Y^s$ and $\CF$ is $C^{k_0+1}$ from $\BBR^2 \times \CX $ to $\CY$ with parameters $\ep$ and $\delta$ where 
\be \begin{split} 
\mathcal{X} &:= H_{\mathrm{e}}^{s+1}(\mathbb{R}) \times (\dot{H}_{\mathrm{e}}^{s}(\mathbb{R}) \cap \dot{H}_{\mathrm{e}}^{1/2}(\mathbb{R})) \times X^s \times \mathbb{R} , \\
 \mathcal{Y} &:= H_{\mathrm{e}}^{s-1}(\mathbb{R}) \times (  \dot{H}_{\mathrm{e}}^{s-1}(\mathbb{R}) \cap \dot{H}_{\mathrm{e}}^{-1/2}(\mathbb{R}))  \times Y^{s-1}. 
 \end{split} \label{localpatch:defspaces} \ee

\subsection{Proof of local bifurcation for the vortex patch} \label{localpatch:couplingsection}

The remainder of the section is devoted to constructing elements of the zero-set of $\mathcal{F}$ with $0 < \epsilon, \delta \ll1$; these correspond to solutions to the vortex patch problem with nontrivial total vorticity.   As a preliminary step, 
let us note that the first two components of $\mathcal{F}$ are formally very similar to the corresponding operator for the point vortex case.  Thus $\CF_{1,2} (0, \delta; 0, \widetilde \psi, \beta, \widetilde c)=0$ and the first two rows of the operator matrix  $D\mathcal{F}(0, \delta; 0, \widetilde \psi, \beta, \widetilde c)$ have invertible entries on the diagonal, and zero entries elsewhere. Moreover, when $\beta=0$, our assumptions on $\gamma$ implies that there exists a unique, radial, negative solution $(\widetilde{F}^*,a^*)$ to the unperturbed problem found by taking $\Gamma = \iota_{B_1}$, the identity function on $B_1$, in \eqref{localpatch:Fsemilineareq}. In particular, the radial symmetry of $\widetilde F^*$ immediately implies $\widetilde H(0) =0$ and thus $\CF(0, 0; 0, 0, 0, \widetilde c)=0$ for any $\widetilde c$. Summarizing these properties we have 

\begin{lemma} \label{localpatch:surfaceproblemlemma} Let $\mathcal{F}$ be defined as in \eqref{localpatch:defF}.  Then $\CF_{1,2} (0, 0; 0, o, 0, \widetilde c)=0$ and 
\[ 
D \mathcal{F}_1(0, 0; 0, 0,0, \widetilde c) = (g -\alpha^2 \partial_{x_1}^2, \, 0, \, 0, 0), \qquad D \mathcal{F}_2(0, 0; 0, 0,0, \widetilde c) = (0, \, \mathcal{G}(0), \, 0, \, 0), 
\]
where $D = (D_{\widetilde \eta}, D_{\widetilde \psi}, D_\beta, D_{\widetilde c})$.  Consequently, at $(0, 0; 0, 0,0, \widetilde c)$, 
\[
D_{\widetilde \eta} \mathcal{F}_1: H^{s+1}_{\mathrm{e}}(\mathbb{R}) \to H^{s-1}_{\mathrm{e}}(\mathbb{R}),
\]
 and 
 \[ 
 D_{\widetilde \psi} \mathcal{F}_2 : \dot{H}_{\mathrm{e}}^{s}(\mathbb{R}) \cap \dot{H}_{\mathrm{e}}^{1/2}(\mathbb{R}) \to \dot{H}_{\mathrm{e}}^{s-1}(\mathbb{R}) \cap \dot{H}_{\mathrm{e}}^{-1/2}(\mathbb{R}) \]
 are invertible. 
\end{lemma}

Because $D_{\widetilde \eta} \mathcal{F}_1$ and $D_{\widetilde \psi} \mathcal{F}_2$ are isomorphisms at $0$, the main obstacle is in showing that $D_\beta \mathcal{F}_3 = D_\beta \big( \p_\theta \widetilde H(\beta) \big)_{\beta=0}$ (at $\delta =0$) is an isomorphism. 
Before proceeding with that analysis, though, we pause to make some useful observations.

Recall that we are seeking solutions  where the vortical region is a perturbation of the ball of radius $\delta > 0$ centered at the origin.  Because $\widetilde{F}^*$ is radial, and the linearized problem around it is non-degenerate, we will see shortly that it is helpful to consider the inverse of the linearized operator in polar coordinates. Assumption \eqref{localpatch:gammaassumptions}, the maximum principle, and the radial symmetry imply 
\be 
\p_r \widetilde F^*(0)=0, \qquad \partial_r \widetilde{F}^* > 0 \quad \textrm{ for } r\in (0, 1]. 
\label{localpatch:F0primepos} \ee
Moreover, by the definition of total vorticity and our normalization,
\be 
1  = \int_{B_{1}} \widetilde{\omega}^* \, dx  = 2\pi \int_0^1 r \Delta \widetilde{F}^* \, dr = 2\pi \int_0^1 \p_r (r \p_r \widetilde F^*  ) \, dr  = 2\pi  \partial_r \widetilde{F}^*(1).  \label{localpatch:Ftildestar1} 
\ee
We define for each $n \geq 1$, 
\be \mathfrak{R}_n(r) := 2 a^* ( \partial_r^2 + \frac{1}{r} \partial_r - \frac{n^2}{r^2} - \gamma^\prime(\frac{1}{a^*}\widetilde{F}^*) )^{-1}  [\gamma(\frac{1}{a^*}\widetilde{F}^*) r^n ]. \label{localpatch:defmathfrakRn} \ee
 
\begin{proposition}\label{localpatch:variationformulaprop} Let $\widetilde{H}$ be defined as in \eqref{localpatch:defPsitilde}.  Let $\dot{\beta} \in X^s$ be given by 
\[ \dot{\beta}(\theta) = \sum_{n\geq 2} \dot{\beta}_n \cos{( \frac{n-1}{2} \pi + n\theta)}.\]
Then for $z = e^{i\theta} \in S^1$, 
\be  
\left[ D_\beta\widetilde{H}(0) \dot{\beta} \right](z) =  \frac{1}{2\pi} \sum_{n \geq 1} [1-\pi \frac {n+1}n  \mathfrak{R}_n^\prime(1)] \dot{\beta}_{n+1} \cos{(\frac{n\pi}{2} + n \theta)}, 
\label{localpatch:variationformula}\ee
where $\mathfrak{R}_n$ are real-valued functions defined in \eqref{localpatch:defmathfrakRn}.\end{proposition}

 \begin{proof}  
 Let $\{ \beta_\lambda\}_{\lambda \geq 0}$ be a family in $X^s$ smoothly parameterized by $\lambda$ such that $\beta_0 \equiv 0$; let $\{ \Gamma_\lambda\}_{\lambda \geq 0}$ be the corresponding family of conformal mappings uniquely determined by \eqref{localpatch:defbeta}.
Notice that this implies $\Gamma_0 = \iota_{B_{1}(0)}$.  Take $\{(\widetilde{F}_\lambda, a_\lambda) \}$ to be a smoothly parameterized family of solutions to  \eqref{localpatch:Fsemilineareq} for $\Gamma_\lambda$, where $(\widetilde{F}_0, a_0) = (\widetilde{F}^*, a^*)$ and put 
\[ 
\widetilde{h}_\lambda := \widetilde{H}(\beta_\lambda), \qquad \widetilde{f}_\lambda :=  \widetilde{F}_\lambda \circ \Gamma_\lambda. 
\]
 Chasing definitions, this means that $F_\lambda$ satisfies 
\[ 
\widetilde{F}_\lambda = \Delta_0^{-1} \left[ a_\lambda |\partial_z \Gamma_\lambda|^2 \gamma(\frac{1}{a_\lambda}\widetilde{F}_\lambda)\right].
\]
Here $\Delta_0^{-1}$ indicates the inverse Laplacian with homogeneous Dirichlet conditions on $B_{1}(0)$.  
  
Let
\[ 
\dot{\beta} := D_\lambda|_{\lambda = 0} \beta_\lambda, \qquad \dot{\Gamma} := D_\lambda|_{\lambda = 0} \Gamma_\lambda, \qquad \dot{\widetilde{F}} := D_\lambda|_{\lambda=0} \widetilde{F}_{\lambda}, \qquad \dot{\widetilde{h}} := 
D_\lambda|_{\lambda=0} \widetilde h_\lambda. 
\]
For each $z \in S^1$,
\begin{align*} 
\dot{\widetilde{h}}(z)& = [D_\lambda \int_{B_1} 
 \frac{1}{2\pi} \log{|\Gamma_\lambda(z) - \Gamma_\lambda(z^\prime)|} a_\lambda | \Gamma_\lambda^\prime(z^\prime)|^2  \gamma(\frac{1}{a_\lambda}\widetilde{F}_\lambda(z^\prime)) \, dz^\prime]\big|_{\lambda =0}. 
\end{align*}
Evaluating the derivative on the right-hand side reveals
 \begin{align}\dot{\widetilde{h}}(z) &= \frac{1}{2\pi} \int_{B_1} \log{|z -z^\prime|} \Delta \dot{\widetilde{F}}(z^\prime) \, dz^\prime \nonumber \\
 & \qquad + \frac{1}{2\pi} \int_{B_1} \frac{ \realpart{\left[(\overline{z -z^\prime})(\dot{\Gamma}(z) - \dot{\Gamma}(z^\prime))\right]}}{|z -z^\prime|^2} \Delta \widetilde{F}^*(z^\prime) \, dz^\prime \nonumber  \\
 & = \frac{1}{2\pi} \int_{B_1} \left( \log{|z -z^\prime|} \Delta \dot{\widetilde{F}}(z^\prime) +  \realpart{[\frac{\dot{\Gamma}(z) - \dot{\Gamma}(z^\prime)}{z-z^\prime}]} \widetilde{\omega}^*(z^\prime) \right) \, dz^\prime \label{localpatch:Psidot0} \\
 & =:\frac{1}{2\pi}( \textrm{I} + \textrm{II}). \nonumber
 \end{align}
 Likewise, the variation $\dot{\widetilde{F}}$ solves the linearized problem
 \be \left\{\begin{array}{ll} 
 \Delta \dot{\widetilde{F}} =  \gamma^\prime( \displaystyle\frac{1}{a^*}\widetilde{F}^*) \dot{\widetilde{F}} + 2 a^* \realpart{[ \dot{\Gamma}^\prime]} \gamma(\displaystyle\frac{1}{a^*}\widetilde{F}^*) 
&  \\ \qquad 
+ \dot{a} [\gamma(\displaystyle\frac{1}{a^*} \widetilde{F}^*) -\frac{1}{a^*} \gamma^\prime(\frac{1}{a^*} \widetilde{F}^*) \widetilde{F}^*] & \textrm{in } B_1 \\
\int_{B_1} \Delta \dot {\widetilde F} dx=0 & \textrm{in } B_1\\
\dot{\widetilde{F}} = 0 & \textrm{on } \partial B_1. \end{array} \right. \label{localpatch:Fdot0eq} \ee
 
 To analyze \eqref{localpatch:Psidot0} we make use of the analyticity of $\Gamma_\lambda$ to expand it as a power series, 
 \[ \Gamma_\lambda = z + \sum_{n \geq 2} a_n(\lambda) z^n.\]
The coefficients depend smoothly on $\lambda$, and thus
   \[ \dot{\Gamma}^\prime = \sum_{n \geq 2} n \dot{a}_n  z^{n-1}.\]
   Inserting this ansatz into the equation satisfied by $\dot{\widetilde{F}}$ \eqref{localpatch:Fdot0eq}, we find
   \be  \begin{split}\Delta \dot{\widetilde{F}} &= \gamma^\prime(\displaystyle\frac{1}{a^*} \widetilde{F}^*) \dot{\widetilde{F}}  + 2 a^* \realpart{ [ \sum_{n\geq 2} n \dot{a}_n z^{n-1} ] } \gamma(\frac{1}{a^*}\widetilde{F}^*) \\
   & \qquad +\dot{a} [ \gamma(\displaystyle\frac{1}{a^*} \widetilde{F}^*) -\frac{1}{a^*} \gamma^\prime(\frac{1}{a^*} \widetilde{F}^*) \widetilde{F}^*]. \end{split}  \ee
 
For simplicity, let us reindex the expansion by taking $\dot{b}_n := \dot{a}_{n+1}$, so that 
\be 
\dot{\Gamma} = \sum_{n \geq 1} \dot{b}_n z^{n+1}, \qquad \dot{\Gamma}^\prime = \sum_{n \geq 1} (n+1) \dot{b}_n z^n.
\label{localpatch:Gammaexpansion2} \ee
Comparing \eqref{localpatch:Gammaexpansion2} to \eqref{localpatch:beta} and \eqref{localpatch:defbeta}, it is clear 
 \be 
 \dot{b}_n 
 = i^n \dot{\beta}_{n+1} .\label{localpatch:symmetry3} \ee
 
 Consider the second term on the right-hand side of \eqref{localpatch:Psidot0}, II, and use the radial symmetry of $\widetilde{\omega}^* = a^*\gamma(\widetilde{F}^*/a^*)$ and the fact that $\int_{B_1} \widetilde \omega^* dx =1$, we obtain 
\be \begin{split}
\text{II}=& \int_{B_1} \realpart{[\frac{\dot{\Gamma}(z) - \dot{\Gamma}(z^\prime)}{z-z^\prime}]} \widetilde{\omega}^*(z^\prime) \, dz^\prime  = \int_{B_1} \realpart{[ \sum_{n\geq 1} \dot{b}_n  \sum_{j=0}^{n} z^j (z^\prime)^{n-j} ]} \widetilde{\omega}^*(z^\prime) \, dz^\prime \\
 & = \realpart{\frac{\dot{\Gamma}(z)}{z}} \int_{B_1} \widetilde{\omega}^*(z^\prime) \, dz^\prime = \realpart{\frac{\dot{\Gamma}(z)}{z}} = \sum_{n\geq 1} \dot{\beta}_{n+1} \cos{(\frac{n\pi}{2} + n\theta)}.
 \label{localpatch:nullspaceident1beta} \end{split}\ee
 
 Next we treat I, the first term in the integrand on the right-hand side of \eqref{localpatch:Psidot0}.  
\[ \begin{split}
\text{I} = &\lim_{r_0 \to 1-} \int_{B_{r_0}} \log{|z-z^\prime|} \Delta \dot{\widetilde{F}}(z^\prime)\, dz^\prime \\
=& \lim_{r_0 \to 1-} \Big(\int_{\partial B_{r_0}} \log{|z-z^\prime|} \nabla_N  \dot{\widetilde{F}}(z^\prime) \, d\sigma(z^\prime)  - \int_{\p B_{r_0}} \nabla_N( \log{|z-z^\prime|})   \dot{\widetilde{F}}(z^\prime) \, d\sigma(z^\prime)\Big)
\end{split} \]
where we integrated in parts twice and 
$\nabla_N$ denotes the outward normal derivative. It is easy to show that the second integral vanishes in the limit $r_0 \to 1$.
Since we are integrating over a ball, $\nabla_N = \partial_r$, and hence the polar representation of the integral for $z^\prime = r_0 e^{i\theta}$ is   
\be  
\textrm{I} = \lim_{r_0 \to 1-} \text{I}(r_0) := \lim_{r_0 \to 1-} \Big(r_0 \int_0^{2\pi} \log{| e^{i\theta} - r_0 e^{i\theta^\prime}|}  (\partial_r \dot{\widetilde{F}})(r_0 e^{i\theta^\prime}) \, d\theta^\prime \Big). 
\label{localpatch:nullspaceident2} \ee

Another way to write \eqref{localpatch:Fdot0eq} is 
\[ \dot{\widetilde{F}} = (\Delta - \gamma^\prime(\frac{1}{a^*} \widetilde{F}^*))^{-1} \left[ 2 a^*\realpart{[ \dot{\Gamma}^\prime ]} \gamma(\frac{1}{a^*}\widetilde{F}^*) + \dot{a}( \gamma(\frac{1}{a^*} \widetilde{F}^*) - \frac{1}{a^*} \widetilde{F}^* \gamma^\prime(\frac{1}{a^*} \widetilde{F}^*) )\right],\]
where the inverse exists by our assumptions on $\gamma$, and is taken with homogeneous Dirichlet boundary conditions.  As many of the quantities involved are radial, it is advantageous to express everything in polar coordinates.  Write
\[ \dot{\widetilde{F}} = \dot{\widetilde{G}}_1 + \dot{\widetilde{G}}_2,\]
where 
\[ 
\dot{\widetilde{G}}_1:=  (\Delta - \gamma^\prime(\frac{1}{a^*} \widetilde{F}^*))^{-1} \left[ 2 a^*\realpart{[ \dot{\Gamma}^\prime ]} \gamma(\frac{1}{a^*}\widetilde{F}^*) \right],
\]
and $\dot{\widetilde{G}}_2 = \dot{\widetilde{F}} - \dot{\widetilde{G}}_1$ is radial. Writing the inverse in polar coordinates and expanding $\Gamma^\prime$ and using \eqref{localpatch:Gammaexpansion2} and \eqref{localpatch:symmetry3},  we find that
\begin{align*} 
\dot{\widetilde{G}}_1(r e^{i\theta}) & =  2a^* \sum_{n\geq 1}  
( \partial_r^2 + \frac{1}{r}\partial_r - \frac{n^2}{r^2} - \gamma^\prime(\frac{1}{a^*} \widetilde{F}^*) )^{-1} \left[ \gamma(\frac{1}{a^*} \widetilde{F}^*) r^n \realpart{( (n+1) \dot{b}_n e^{in\theta} )} \right] \\
& = \sum_{n\geq 1}   \mathfrak{R}_n(r) (n+1) \dot{\beta}_{n+1} \cos{(\frac{n\pi}{2} + n \theta)}. 
\end{align*}

Thus, for $z^\prime \in \partial B_{r_0}$ with $\arg{z} = \theta^\prime$, 
\[ 
(\partial_r \dot{\widetilde{G}}_1)(z^\prime)  = \sum_{n\geq 1}   \mathfrak{R}^\prime_n(r_0) (n+1) \dot{\beta}_{n+1} \cos{(\frac{n\pi}{2} + n \theta^\prime)}. 
\]
Inserting this expression into \eqref{localpatch:nullspaceident2}, we find that
\begin{align*}
\textrm{I} (r_0) & = \sum_{n \geq 1} \mathfrak{R}_n^\prime(r_0) r_0 (n+1) \dot{\beta}_{n+1}  \int_0^{2\pi} \log{|1- r_0 e^{i(\theta^\prime -\theta)}|} \cos{(\frac{n\pi}{2} + n \theta^\prime)} \, d\theta^\prime \nonumber \\
& \qquad + r_0 \partial_r \dot{\widetilde{G}}_2(r_0) \int_0^{2\pi} \log{|1- r_0 e^{i(\theta^\prime -\theta)}|}   \, d\theta^\prime \nonumber \\
& = -\sum_{n \geq 1} \mathfrak{R}_n^\prime(r_0) r_0 (n+1) \dot{\beta}_{n+1}  \int_0^{2\pi} \realpart{[\sum_{j \geq 1}  \frac{1}{j} r_0^j e^{ij(\theta^\prime - \theta)} \cos{(\frac{n\pi}{2} + n\theta^\prime)}]} \, d\theta^\prime   \nonumber \\
& = - \pi \sum_{n \geq 1} \frac{n+1}n \mathfrak{R}_n^\prime(r_0) \dot{\beta}_{n+1}  r_0^{n+1}  \cos{(\frac{n\pi}{2} +n\theta)}.  
  \end{align*}
Taking $r_0 \to 1$, this shows that 
\be \textrm{I} =  -\pi \sum_{n \geq 1} \frac {n+1}n \mathfrak{R}_n^\prime(1)  \dot{\beta}_{n+1} \cos{(\frac{n\pi}{2}+n\theta)} \label{localpatch:nullspaceident3}  \ee
  Finally, combining \eqref{localpatch:Psidot0}, \eqref{localpatch:nullspaceident1beta}, \eqref{localpatch:nullspaceident2}, and \eqref{localpatch:nullspaceident3}, we arrive at \eqref{localpatch:variationformula}.   This completes the proof.  
 \end{proof}
 
\begin{lemma}[Fredholm map] \label{localpatch:fredholmlemma} Let $\widetilde{H}$ be defined as in  \eqref{localpatch:defPsitilde}.  Then the following statements are true.  
\begin{itemize}
\item[(a)]  $\mathscr{N}(D_\beta \widetilde{H}(0))$ is one-dimensional and generated by the map $\theta \mapsto \sin{(2\theta)}.$
\item[(b)] $\mathscr{R}(\partial_\theta D_\beta \widetilde{H}(0)) = (\theta \mapsto \cos \theta)^\perp $ is codimension $1$ in $Y^{s-1}$.
\item[(c)] $\partial_\theta D_\beta \widetilde{H}(0) \in L(X^s, Y^{s-1})$ is a Fredholm operator of index $0$.  
\end{itemize}
   \end{lemma}  
   
 \begin{proof}  (a) In Proposition \ref{localpatch:variationformulaprop}  a formula was derived for the variations of $\widetilde{H}$ in $\beta$.  We now proceed to identify its null space, which will require a closer look at $\mathfrak{R}_n$.   Adopting the same notation as in the previous lemma, let us additionally define for $n\ge 1$
\[ 
q_n(r) := a^* \left( \partial_r^2+ \frac{1}{r} \partial_r - \frac{n^2}{r^2} - \gamma^\prime(\frac{1}{a^*}\widetilde{F}^*) \right)^{-1} [\gamma(\frac{1}{a^*}\widetilde{F}^*) r^n] = \frac{1}{2} \mathfrak{R}_n(r).\]
 Explicitly, $q_n$ solves
 \be \left\{ \begin{array}{ll}\displaystyle \frac{1}{r} \partial_r ( r \partial_r q_n) -( \gamma^\prime(\frac{1}{a^*}\widetilde{F}^*) + \frac{n^2}{r^2} ) q_n = a^*\gamma(\frac{1}{a^*}\widetilde{F}^*) r^n & \textrm{ on } (0,1)\\
 &\\
 q_n(0) = q_n(1) = 0. & \end{array} \right. 
\label{localpatch:nullspaceqneq} \ee
While $q_n(1) =0$ comes from the definition of the inverse, the vanishing of $q_n$ at the origin is required in order for $q_n$ to be smooth at $r=0$. These conditions imply $q_n \in C^N$. Moreover, a blow-up analysis at $r=0$, an ODE regular singular point, reveals $q_n'(0) =0$ for $n\ge 2$. \\
 
\emph{Claim I:}  $q_n$ is non-positive on $(0,1)$. In particular, there is no interval $(r_1, r_2) \subset (0, 1)$  such that 
 \be q_n \geq 0 \textrm{ on } (r_1, r_2), \qquad q_n(r_1) = q_n(r_2) = 0. \label{localpatch:hypothesisclaim1} \ee
 Seeking a contradiction, suppose that there exists  $0 \leq r_1 < r_2 \leq 1$ as above.  Then from \eqref{localpatch:nullspaceqneq} we compute
 \begin{align} \int_{r_1}^{r_2} r^{n+1} \widetilde{\omega}^*  \partial_r \widetilde{F}^*  \, dr &= \int_{r_1}^{r_2} \left[ \partial_r (r \partial_r q_n) -r(\gamma^\prime(\frac{1}{a^*}\widetilde{F}^*) + \frac{n^2}{r^2}) q_n) \right] \partial_r \widetilde{F}^* \, dr \nonumber\\
 & = -\int_{r_1}^{r_2}  \left[ r (\partial_r^2 \widetilde{F}^*) \partial_r q_n + r q_n  (\partial_r \widetilde{F}^*) ( \gamma^\prime(\frac{1}{a^*}\widetilde{F}^*)  + \frac{n^2}{r^2}) \right] \, dr \nonumber  \\
 & \qquad + \left( r \partial_r \widetilde{F}^* \partial_r q_n\right)\bigg|_{r_1}^{r_2}  \nonumber \\
 & = \int_{r_1}^{r_2} \left[ \partial_r( r \partial_r^2 \widetilde{F}^*) - r (\partial_r \widetilde{F}^*) ( \gamma^\prime(\frac{1}{a^*}\widetilde{F}^*) + \frac{n^2}{r^2}) \right] q_n \nonumber \\ 
 & \qquad + ( r \partial_r q_n \partial_r \widetilde{F}^*) \bigg|_{r_1}^{r_2}. \label{localpatch:nullspaceint1} \end{align} 
 Here the fact that $q_n$ vanishes at the endpoints has been used in the second integration by parts to cancel the boundary terms.   Recalling that $\widetilde{F}^*$ is radial and satisfies \eqref{localpatch:Fsemilineareq} with $\Gamma = \iota_{B_1}$, we see that 
 \begin{align*} \partial_r (r \partial_r^2 \widetilde{F}^*) &= \partial_r ( r a^* \gamma(\frac{1}{a^*}\widetilde{F}^*) - \partial_r \widetilde{F}^*) = a^*\gamma(\frac{1}{a^*}\widetilde{F}^*) + r \gamma^\prime(\frac{1}{a^*}\widetilde{F}^*) \partial_r \widetilde{F}^* - \partial_r^2 \widetilde{F}^* \\
 & = \frac{1}{r} \partial_r( r \partial_r \widetilde{F}^*) +r \gamma^\prime(\frac{1}{a^*}\widetilde{F}^*) \partial_r \widetilde{F}^* - \partial_r^2 \widetilde{F}^* \\
 & = \frac{1}{r} \partial_r \widetilde{F}^*+ r \gamma^\prime(\frac{1}{a^*}\widetilde{F}^*) \partial_r \widetilde{F}^*. \end{align*}
 Inserting this identity into \eqref{localpatch:nullspaceint1}, gives 
\[ \int_{r_1}^{r_2} r^{n+1} \widetilde{\omega}^*  \partial_r \widetilde{F}^*  \, dr = \int_{r_1}^{r_2} q_n \frac{1-n^2}{r} \partial_r \widetilde{F}^* \, dr + ( r \partial_r q_n \partial_r \widetilde{F}^*) \bigg|_{r_1}^{r_2}. \]
 This is impossible: the integrand on the right-hand side is non-positive by \eqref{localpatch:hypothesisclaim1} and \eqref{localpatch:F0primepos}, yet $\partial_r q_n(r_2) \leq 0$, and $\partial_r q_n(r_1) \geq 0$ so that the second term is non-positive.  We have proved the first claim.  \\
 
 \emph{Claim II:}  The map $\dot{\beta}: \theta \mapsto \sin{(2\theta)}$ is in the null space of $\dot{\widetilde{H}} := D_\beta \widetilde{H}(0)$.    First observe that setting $r_1 = 0$, $r_2 = 1$ and proceeding as in the proof of the first claim leads to the identity 
 \be \int_{0}^{1} r^{n+1} \widetilde{\omega}^*  \partial_r \widetilde{F}^*  \, dr = \int_0^{1} q_n \frac{1-n^2}{r} \partial_r \widetilde{F}^* \, dr +  (\partial_r q_n)(1) (\partial_r \widetilde{F}^*)(1). \label{localpatch:identityclaim2} \ee
 In particular, when $n = 1$, \eqref{localpatch:identityclaim2} implies 
 \begin{align*} \partial_r \widetilde{F}^*(1) \partial_r q_1(1) & =  \int_0^{1} r^2 \widetilde{\omega}^* \partial_r \widetilde{F}^* \, dr \\
 & = \int_0^{1} r^2 (\frac{1}{r} \partial_r (r \partial_r \widetilde{F}^*)) \partial_r \widetilde{F}^* \, dr \\
 & = \frac{1}{2} \left(  \partial_r \widetilde{F}^*(1)\right)^2 .\end{align*}
 Since $\partial_r \widetilde{F}^*(1) \neq 0$, we may divide to find
 \be \partial_r q_1(1) = \frac{1}{2} \partial_r \widetilde{F}^*(1). \label{localpatch:identity2claim2} \ee
 Combining \eqref{localpatch:identity2claim2} and \eqref{localpatch:Ftildestar1} gives simply that
 \[ \mathfrak{R}_1^\prime(1) = 2 q_1^\prime(1) =  \partial_r \widetilde{F}^*(1) = \frac{1}{2\pi}.\]
 Now, if $\dot{\beta}(\theta) = \sin{(2\theta)}$, then the coefficients are given by $\dot{\beta}_2 = -1$ and $\dot{\beta}_n = 0,~n \neq 2,$
and thus from \eqref{localpatch:variationformula} we compute
\[ 
\dot{\widetilde{H}} (\theta) = - \frac{1}{2\pi} [1 - 2\pi \mathfrak{R}_1^\prime(1)] \sin{\theta},
\]
which vanishes identically.  Hence $\theta \mapsto \sin{(2\theta)} \in \mathscr{N}(D_\beta \widetilde{H}(0))$, as claimed.  \\
 
 \emph{Claim III:}  Let $\dot{\beta} \in X^s$ be given by
 \[ 
 \dot{\beta}: \theta \mapsto \cos{(\frac{n}{2}\pi+(n+1)\theta)}.
 \]
  Then $\dot{\beta}$ is not in the null space of $D_\beta \widetilde{H}(0)$ for any $n  > 1$.  To prove this, consider again the identity \eqref{localpatch:identityclaim2}.   By the first claim, $q_n$ is non-positive on $(0,1)$, and so
\[    \partial_r q_n(1) \p_r \widetilde{F}^*(1)  \leq \int_0^{1}  r^{n+1} \widetilde{\omega}^* \partial_r  \widetilde{F}^* \, dr. \]
Now, using the PDE satisfied by $ \widetilde{F}^*$, 
\begin{align*} \int_{0}^{1} r^{n+1} \widetilde{\omega}^* \partial_r  \widetilde{F}^* \, dr & = \int_0^{1} r^{n+1} \left( \partial_r^2  \widetilde{F}^* + \frac{1}{r} \partial_r  \widetilde{F}^*\right) \partial_r  \widetilde{F}^* \, dr \\
& < \frac{1}{2}  \int_0^{1} \left( 2 r^{n+1} \partial_r  \widetilde{F}^* \partial_r^2  \widetilde{F}^* + (n+1) r^n (\partial_r  \widetilde{F}^*)^2 \right) \, dr \\
& = \frac{1}{2} (\partial_r  \widetilde{F}^*(1))^2.
\end{align*}
Thus, 
\[ \partial_r q_n(1) < \frac{1}{2}  \partial_r  \widetilde{F}^*(1) = \frac{1}{4\pi}.\]
As an immediate consequence, we have that 
\be 
1\ge 1-\pi \frac{n+1}n \mathfrak{R}_n^\prime(1) = 1-2\pi \frac{n+1}n \partial_r q_n(1)> \frac {n-1}{2n} \ge \frac 14, \qquad n\ge2. 
\label{localpatch:identity1claim3} \ee

Using the variation formula \eqref{localpatch:variationformula}, we compute that 
\[ 
\left[ D_\beta \widetilde{H}(0) \dot{\beta} \right](e^{i\theta})    = \frac{1}{2\pi}[1-2\pi \mathfrak{R}_n^\prime(1)]  \cos{(\frac{n\pi}{2}+n\theta)}.
\]
This function cannot vanish identicaly, since the bracketed quantity is nonzero by \eqref{localpatch:identity1claim3}.  It follows that no $\dot{\beta}$ of this form can be in the null space.  This completes the proof of the third claim. 
 
Part (a) is a direct consequence of Claims II and III.   Next consider (b).  From \eqref{localpatch:variationformula} we see that for any $\dot{\beta} \in X^s$, 
\be 
\left[\partial_\theta D_\beta \widetilde{H}(0)\dot{\beta}\right] = \frac 1{2\pi} \sum_{n\geq 1} n \dot{\beta}_{n+1} [ \pi \frac{n+1}n \mathfrak{R}_n^\prime(1)-1]  \cos{(\frac{n-1}{2}\pi + n \theta)}. 
\label{localpatch:variationformula2}\ee
Part (c) follows immediately and the proof of the lemma is compete. 
 \end{proof}
 
We are now in a position to prove the main theorem for the vortex patch problem, namely the existence of small-amplitude solutions, via Lyapunov-Schmidt reduction.
 
\begin{proof}[Proof of Theorem \ref{localpatch:bifurcationtheorem}] 
Let 
\[\begin{split}
&\CX_0 =\{\big(0, 0, \tilde \tau \sin(2\theta), \widetilde c \big) \, : \, \tau, \widetilde c \in \BBR\}, \quad \CX_1 = \CX_0^\perp \subset \CX\\ 
&\CY_0 = span\{(0, 0, \cos \theta) \}, \quad \CY_1 = \CY_0^\perp \subset \CY
\end{split}\]
and $\Pi_\CX$ and $\Pi_\CY$ be the associated projections to $\CX_0$ and $\CY_0$, respectively. 
From the Implicit Function Theorem and Lemmas \ref{localpatch:surfaceproblemlemma} and \ref{localpatch:fredholmlemma}, there exists a neighborhood $\mathcal{U}$ of the line $\{(\ep =0, \delta=0; \tilde \tau=0, \widetilde c)\} \subset \BBR^4$ and a $C^{k_0+1}$ mapping ($k_0+1$ given in Theorem \ref{localpatch:bifurcationtheorem}) 
\[
u= (u_1, u_2, u_3, \widetilde c) : \mathcal{U} \to \CX_1, \qquad u(0,0,0, \widetilde c) = (0,0,0, \widetilde c)
\]
such that in a neighborhood of the line $\{(\ep =0, \delta=0; \widetilde \eta=0, \widetilde \psi=0, \beta=0, \widetilde c)\} \subset \BBR^2 \times \CX$ 
\[
\CF =0 \quad \text{ iff } \quad \widetilde \eta = u_1(\ep, \delta, \tilde \tau, \widetilde c), \quad \widetilde \psi= u_2(\ep, \delta, \tilde \tau, \widetilde c), \quad \beta = \tilde \tau \sin (2\theta) + u_3 (\ep, \delta,\tilde  \tau, \widetilde c).
\]
What is left to achieve is to find a solution of 
\be \label{localpatch:bifurE}
0= h(\ep, \delta, \tau, \widetilde c) := \frac 1{\delta\pi} \int_0^{2\pi} \cos{(\theta)} \, \CF_3 \big(\ep, \delta; u(\ep, \delta, \delta \tau, \widetilde c)+(0,0, \delta \tau \sin (2\theta), 0) \big) \,  d\theta.
\ee
Note we have built the scaling $\tilde \tau=\delta \tau$ and $\frac 1\delta \CF_3$ into the definition of $h$. In fact, analyzing these nonlinear functionals very carefully, the explicit form of $\CF_{1,2}$ implies that $\frac 1\ep u_{1,2}$ are also $C^{k_0}$ mapping defined on $\mathcal{U}$. Rescale $\tilde \tau = \delta \tau$, the explicit form of $\CF_3$ also implies $\frac 1\delta u_3$ is a $C^{k_0}$ mapping in $\ep$, $\delta$, $\tau$, and $\widetilde c$. Substituting these observations into $\CF_3$, it is straight forward to compute that $h$ is a $C^{k_0}$ function with the expansion  
\[
h(\ep, \delta, s, \widetilde c) = \widetilde c + \frac 1{4\pi} + h_1(\ep, \delta, s, \widetilde c), \qquad h_1 = O(\delta + \ep).
\]
Applying the Implicit Function again, we find that for small $(\ep, \delta, s)$, there exists a solution $\widetilde c = -\frac 1{4\pi} + O(\ep + \delta)$ which is $C^{k_0}$ in $\ep$, $\delta$, and $s$. Substituting this back to $\CF_3$ again, we also obtain $u_3 = O(\delta^2 + \delta \ep)$. 
\end{proof}

\section{Global bifurcation with a point vortex} \label{gobal:pointvortexsection}

\subsection{Introduction and notation} 

Consider the steady point vortex problem (PtV) in the periodic setting (Per) which corresponds to equations  \eqref{intro:bernoullieq}, \eqref{intro:kinematiceq}, \eqref{localpoint:defbper}, and \eqref{setup:ceq-per}.  Unlike in deriving small amplitude solutions in Section \ref{local:pointvortexsection}, we can not work with the $\ep$-scaled unknowns for the global bifurcation problem and it is more natural to work with the original unknowns $(\eta, \psi, c)$, where the graph of $\eta$ defines the air-water interface, $\psi$ is the trace on the surface of the harmonic part of the stream function, and $c$ is the wave speed.  We have a point vortex situated at the origin with vortex strength $\epsilon$.  The abstract problem is to find $(\epsilon,\eta, \psi, c)$ in the zero-set of the operator 
\[ \mathcal{F} = (\mathcal{F}_1, \mathcal{F}_2, \mathcal{F}_3) : \mathbb{R} \times \mathcal{O} \to Y,\]
where 
\[ \begin{split}
&X := {H}_{\textrm{m}}^k(L\BBS^1) \times {H}_{\textrm{m}}^k(L\BBS^1) \times \mathbb{R}, \qquad \mathcal{O} = \{(\eta, \psi, c) \in X\, :\, \eta(0) >-1\}, \\ 
&Y := H_{\textrm{m}}^{k-2}(L\BBS^1) \times H_{\textrm{m}}^{k-2}(L\BBS^1) \times \mathbb{R}, \qquad k\ge 3
\end{split}\]
and 
\be \label{global:defF} \begin{split}
\mathcal{F}_1(\epsilon; \eta, \psi, c) & := c\left( \mathcal{G}(\eta) \psi + \epsilon( -\eta^\prime, 1) \cdot \nabla \CG  \right) + \frac{1}{2} \left( \mathcal{G}(\eta) \psi + \epsilon( -\eta^\prime, 1) \cdot \nabla \CG \right)^2 \\
& \qquad - \frac{1}{2(1+(\eta^\prime)^2)} \left( \psi' - \eta^\prime \mathcal{G}(\eta) \psi + \epsilon (1+(\eta^\prime)^2) \partial_{x_1} \CG \right)^2 \\
& \qquad \qquad + g(1+ \eta) + \alpha^2 \kappa(\eta) - b(\epsilon,\eta, \psi, c) \\
\mathcal{F}_2(\epsilon; \eta, \psi, c) & := - c\p_{x_1}^2 \eta - \p_{x_1}^2 \psi - \epsilon \p_{x_1} \big( (1, \eta^\prime) \cdot \nabla \CG \big)  \\
\mathcal{F}_3(\epsilon; \eta, \psi, c) & := c+(\partial_{x_2} \psi_{\mathcal{H}})(0) + \frac{\ep}{4\pi} \sum_{k=-\infty}^\infty \frac 1{k^2 \pi^2 L^2+1}.
\end{split} \ee
Note that we applied the operator $-\p_{x_1}$ to \eqref{intro:kinematiceq} to define $\CF_2$, which is possible since we work on functions with zero mean on $\BBS^1$. The advantage of this procedure is to make $D_\psi \CF_2$ a second order positive operator like $D_\eta \CF_1$.

With our choice of spaces, 
 $\mathcal{F}$ is a $C^\infty$ nonlinear mapping as long as the free surface does not contact the point vortex (in fact, $\mathcal{F}$ is analytic (cf. \cite[Lemma 2.2]{craig2000traveling}).   Up to this point, it has been shown that there exists a curve of solutions $\mathscr{C}_{\textrm{loc}}$ parameterized by the vorticity strength $\epsilon$ bifurcating from the trivial solution $\epsilon = 0$.  In this section we endeavor to continue $\mathscr{C}_{\textrm{loc}}$ globally.  With that in mind, define $\mathcal{Z} \subset \mathbb{R} \times \mathcal{O}$ to be the zero set of $\mathcal{F}$ and let $\mathscr{C}$ be the connected component of $\mathcal{Z}$ containing the local curve $\mathscr{C}_{\textrm{loc}}$. Our main theorem is the following:
\begin{samepage}

\begin{theorem}[Global bifurcation]  \label{global:globalift} One of the following alternatives must hold:
\begin{itemize}
\item[(i)] $\mathscr{C}$ is unbounded in $\mathbb{R} \times X$;  
\item[(ii)] there exists a nontrivial irrotational (i.e. $\ep=0$) traveling wave solution in $\mathcal{O}$ where the velocity at the original is equal to the wave speed $c$;
or 
\item[(iii)] along some sequence in $\{(\epsilon_n, \eta_n, \psi_n, c_n)\} \subset \mathscr{C}$, we have $\eta_n(0) \to -1$. 
\end{itemize}  
\end{theorem} \end{samepage}


Clearly Theorem \ref{global:pointvortextheorem} is an immediate consequence of the above statement, and so we need only concern ourselves with the proof of Theorem \ref{global:globalift}.  The tool we elect to use is the global implicit function theorem of Kielh\"ofer, which we now paraphrase.   
\begin{definition}  Let $X$ and $Y$ be (real) Banach spaces such that $X$ is continuously embedded in $Y$.  
\begin{itemize}
\item[(a)] A linear map $A \in \mathcal{L}(X,Y)$ is said to be \emph{admissible} provided that 
\begin{itemize}
\item[(i)] $A$ is a Fredholm operator of index 0, 
\item[(ii)] $A: Y \to Y$ with domain of definition $D(A) = X$ is closed, and 
\item[(iii)] the spectrum of $A$ in a strip $(-\infty, a) \times (-i\epsilon, i\epsilon) \subset \mathbb{C}$ for some $a > 0$ consists of finitely many eigenvalues of finite algebraic multiplicity.  Moreover, their total number is stable under small perturbations in the class of $\mathcal{L}(X,Y)$. \end{itemize}
\item[]
\item[(b)] An operator $F = F(\lambda, x)  \in C^2(\mathbb{R} \times \mathcal{O},  Y)$, where $\mathcal{O} \subset X$ is open, is said to be \emph{admissible} provided that 
\begin{itemize}
\item[(i)] $D_x F(\lambda_0, x_0)$ is admissible in the sense of (a) for each $(\lambda_0, x_0) \in \mathbb{R} \times \mathcal{O}$, and
\item[(ii)] $F$ is a proper map: for every compact $U \subset  Y$, and $K \subset X$ closed and bounded, $K \cap F^{-1}(U)$ is compact in $\mathbb{R} \times X$.  
\end{itemize}
\end{itemize}
\label{global:admissibledef} 
\end{definition}
When $F$ is admissible in the sense above, a degree theoretic global continuation argument (cf. \cite{rabinowitz1971some}) can be applied to the local curve furnished by the implicit function theorem.  Intuitively, the idea is that the principal part of the operator has a compact inverse, allowing a generalization of the classical Leray--Schauder degree.    The result is a ``global implicit function theorem.'' 

\begin{theorem}\emph{(\cite[Theorem II.6.1]{kielhofer2004bifurcation})} \label{global:kielhofertheorem}  Consider an admissible operator $F = F(\lambda,x) : \mathbb{R} \times \mathcal{O} \to Y$.  Suppose that there exists a solution 
\[ F(\lambda_0, x_0) = 0 \]
at which  $D_x F(\lambda_0, x_0)$ is an isomorphism. 
Let $\mathcal{Z}$ denote the zero-set of $F$ and let $\mathscr{C}$ be the connected component of $\mathcal{Z}$ that contains the local solution curve through $(\lambda_0, x_0)$ furnished by the implicit function theorem.  Then one of the following alternatives must hold:
\begin{itemize}
\item[(i)] $\mathscr{C} = \{ (\lambda_0, x_0) \} \cup \mathscr{C}^+ \cup \mathscr{C}^-$, where $\mathscr{C}^+ \cap \mathscr{C}^- = \emptyset$, and $\mathscr{C}^+$, $\mathscr{C}^-$ are each unbounded in $\mathbb{R} \times X$ or with zero distance to $\BBR \times \mathcal{O}^c$; or
\item[(ii)] $\mathscr{C} \setminus \{(\lambda_0, x_0)\}$ is a connected set. 
\end{itemize} 
\end{theorem}

We will apply Theorem \ref{global:kielhofertheorem} to obtain Theorem \ref{global:globalift} so long as we are able to confirm that the operator $\mathcal{F}$ defined by \eqref{global:defF} meets the admissibility criteria of Definition \ref{global:admissibledef}.   This task we address in the next section.

\subsection{Admissibility of $\mathcal{F}$} 

The regularity of $\mathcal{F}$ is only guaranteed so long as the point vortex does not contact the free surface, and hence working near the boundary of $\mathcal{O}$ may pose some technical inconvenience.  
Instead, fix $\delta > 0$, and put
\[ \mathcal{O}_\delta := \{ (\epsilon, \eta, \psi, c) \in \mathbb{R} \times \mathcal{O} : \textrm{dist}(0, \partial{\Omega(\eta)}) > \delta\}, \qquad \mathcal{Z}_\delta := \mathcal{Z} \cap  \mathcal{O}_\delta.\]
Let $\mathscr{C}_\delta$ be the connected component of $Z_\delta$ containing $\mathscr{C}_{\textrm{loc}}$.  All of our work in this section will be with $\mathcal{O}_\delta$ and $\mathscr{C}_\delta$.  Indeed, it is precisely this restriction that is responsible for introducing the third alternative in Theorem \ref{global:globalift}; the first and second correspond to those in Theorem \ref{global:kielhofertheorem}.

\begin{lemma}[Proper map]
 Let $U \subset Y$ be compact and $K \subset \BBR \times \overline{\mathcal{O}_\delta}$ be closed and bounded.  Then $K\cap \mathcal{F}^{-1}({U})$ is compact in $\mathbb{R} \times X$.   
 \label{global:propermaplemma} \end{lemma}

\begin{proof}
Let $U$ and $K$ be given as above.  Put $V :=  K \cap \mathcal{F}^{-1}({U})\subset \BBR \times \mathcal{O}$ and let $ v_n := \{(\epsilon_n, \eta_n, \psi_n, c_n)\}$ be a bounded sequence in $V$.  By construction, 
\[ 
u_n = (f_n, g_n, h_n) := \mathcal{F}(\epsilon_n, \eta_n, \psi_n, c_n) \subset {U}
\]
is bounded in $Y$.  Compactness of ${U}$ ensures that (modulo a subsequence) $\{u_n\}$ converges to some $u$ in $Y$.  In order to prove the lemma, we must use this information to construct a convergent subsequence of $\{v_n \}$ in $V$.  Since $\epsilon_n, c_n \in \mathbb{R}$, it is obvious that along a subsequence $\epsilon_n \to \epsilon$, $c_n \to c$, for some $\epsilon, c \in \mathbb{R}$.

>From \eqref{global:defF} and the definitions of $u_n$, $v_n$, we have
\begin{align} 
f_n & = c_n\left( \mathcal{G}(\eta_n) \psi_n + \epsilon_n (-\eta_n^\prime, 1) \cdot\nabla \CG \right) + \frac{1}{2} \left( \mathcal{G}(\eta_n) \psi_n + \epsilon_n (-\eta_n^\prime, 1) \cdot\nabla \CG\right)^2 \nonumber \\
& \qquad - \frac{1}{2(1+(\eta_n^\prime)^2)} \left( \psi_n' - \eta_n^\prime \mathcal{G}(\eta_n) \psi_n + \epsilon_n (1+(\eta_n^\prime)^2) \partial_{x_1} \CG \right) \nonumber\\
& \qquad \qquad + g(1+\eta_n) -b(\epsilon_n, \eta_n, \psi_n, c_n) + \alpha^2 \kappa(\eta_n)  \nonumber \\
& =:  M(\epsilon_n,\psi_n, \eta_n, c_n) -\alpha^2 \frac{\eta_n^{\prime\prime}}{\left( 1+(\eta_n^\prime)^2\right)^{3/2}},   \end{align}
where $M$ is a bounded quantity in $Y$.   Denoting
\[ 
M_n =M(\epsilon_n,\psi_n, \eta_n, c_n), \qquad M_{mn} := M_m - M_n, \qquad \kappa_{mn} := \kappa(\eta_m) - \kappa(\eta_n),
\]
we see from above that 
\[ f_{mn} := f_m - f_n = M_{mn} + \alpha^2 \kappa_{mn}.\]
Since $\{ f_n\}$ is convergent in $H^{k-2}(L\BBS^1)$ by assumption,  we know 
\be \partial_{x_1}^{k-2} f_{mn} \to  0 \qquad \textrm{in} \qquad L^2(L\BBS^1),\label{global:properfmn} \ee
Also, since $\{M_{n}\}$ is bounded in $H^{k-1}(L\BBS^1)$, the sequence $\{ \partial_{x_1}^{k-2} M_{n}\}$ is bounded in $H^1(L\BBS^1)$.   Morrey's inequality implies  $\{ \partial_{x_1}^{k-2} M_{n}\}$ is bounded in $C^{0,1/2}(L\BBS^1)$, and hence it is precompact in $C^{0,\epsilon}(L\BBS^1)$, for any $\epsilon \in [0, 1/2)$.  This allows us to extract a convergent subsequence in $C^{0,1/4}(L\BBS^1)$, say, which is in turn convergent in $L^2(L\BBS^1)$.  It follows that  
\be \partial_{x_1}^{k-2} M_{mn} \to  0 \qquad \textrm{in} \qquad L^2(L\BBS^1).\label{global:properNmn} \ee

From \eqref{global:properfmn} and \eqref{global:properNmn}, we infer that $\partial_{x_1}^{k-2} \kappa_{mn}$ converges to 0 in $L^2(L\BBS^1)$.  Meanwhile, because  $\{ \eta_n^\prime \}$ is bounded in $H^{k-1}(L\BBS^1)$, 
Up to a subsequence, therefore, $\jbracket{\eta_n^\prime}$ converges in $H^s (L\BBS^1)$, $s<k-1$.  Since $k\ge 3$, this observation permits the following estimate
\begin{align*} 
| \eta_n^{\prime\prime} - \eta_m^{\prime\prime} |_{H^{k-2}} &= |\kappa_{m} \jbracket{\eta_m^\prime}^{3} - \kappa_{n} \jbracket{\eta_n^\prime}^{3} |_{H^{k-2}} \\
& \leq | \kappa_{mn}|_{H^{k-2}} |\jbracket{\eta_{m}^\prime}^3|_{H^{k-2}} + |\kappa_n|_{H^{k-2}} |\jbracket{\eta_m^\prime}^3- \jbracket{\eta_n^\prime}^3|_{H^{k-2}} \to 0
\end{align*}
We have therefore proved that 
$\{ \eta_n \}$ has a convergent subsequence in $H^k(L\BBS^1)$.

Next we consider $\{ \psi_n\}$.    The definition of $\mathcal{F}$ in \eqref{global:defF} yields
\[ 
g_n = - c_n \eta_n'' - \psi_n'' - \epsilon_n \p_{x_1}\big( (1, \eta_n^\prime) \cdot \nabla \CG\big).
\]
We are given that $\{ g_n\}$ is convergent in $H^{k-2}(L\BBS^1)$, and we have at this point demonstrated that $\{ \eta_n \}$ is precompact in $H^{k}(L\BBS^1)$, as are $\{ c_n \}$ and $\{\epsilon_n\}$ in $\mathbb{R}$.  The precompactness of $\psi_n \in H_m^k$ follows immediately,
which completes the lemma.   
\end{proof}

We now compute the Fr\'echet derivative of the operator $\mathcal{F}(\epsilon; \eta, \psi, c)$ and proceed to prove the remaining admissibility properties.  Recall $\CG$, $\nabla \CG$, etc. are always evaluated at $x_2 = 1+ \eta(x_1)$, we have at $(\epsilon; \eta, \psi, c)$,
\be \label{global:DF1} 
\begin{split}  
\mathcal{F}_{1\eta}
\zeta & = \big(c + \mathcal{G}(\eta) \psi + \epsilon( -\eta^\prime, 1) \cdot \nabla \CG\big)\Big(\langle \mathcal{G}_\eta (\eta) \zeta, \psi\rangle + \ep (-\zeta^\prime, 0) \cdot \nabla \CG \\
&\qquad + \ep \zeta (-\eta', 1)\cdot \nabla \p_{x_2} \CG  \Big) 
+ g \zeta + \alpha^2  \kappa_\eta(\eta)\zeta - b_{\eta}(\epsilon,\eta, \psi, c) \zeta \\
& \qquad + \frac{\eta^\prime \zeta^\prime}{ \jbracket{\eta^\prime}^4} \left( \psi' - \eta^\prime \mathcal{G}(\eta) \psi + \epsilon (1+(\eta^\prime)^2) \partial_{x_1} \CG \right)^2   \\
& \qquad - \frac{1}{\jbracket{\eta^\prime}^2} \left( \psi' - \eta^\prime \mathcal{G}(\eta) \psi + \epsilon (1+(\eta^\prime)^2) \partial_{x_1} \CG
\right) \\
& \qquad \qquad  \cdot \left(-\zeta' \mathcal{G}(\eta) \psi - \eta' \langle \mathcal{G}_\eta (\eta) \zeta, \psi\rangle + 2\ep \eta' \zeta' \p_{x_1} \CG + \ep (1+ (\eta')^2) \zeta \p_{x_1 x_2} \CG  
\right) \\
\mathcal{F}_{1\psi}
\phi & = \big(c + \mathcal{G}(\eta) \psi + \epsilon( -\eta^\prime, 1) \cdot \nabla \CG\big) \mathcal{G}(\eta) \phi 
- b_{\varphi}(\epsilon,\eta, \psi, c) \phi\\
& \qquad  - \frac{1}{\jbracket{\eta'}^2} \left( \psi' - \eta^\prime \mathcal{G}(\eta) \psi + \epsilon (1+(\eta^\prime)^2) \partial_{x_1} \CG
\right) \left( \phi' - \eta' \mathcal{G}_\eta(\eta) \phi
\right) \\
\mathcal{F}_{1c}
& = \mathcal{G}(\eta) \psi + \epsilon( -\eta^\prime, 1) \cdot \nabla \CG 
- b_{c}(\epsilon,\eta, \psi, c).
\end{split}\ee
Here the terms $Db(\cdots)(\cdot)$ are only the average of the remaining terms in the expressions. 
\be \label{global:DF2} 
\begin{split}
\mathcal{F}_{2\eta}
\zeta & = -c \zeta'' 
- \epsilon\p_{x_1} \big((0, \zeta^\prime) \cdot \nabla \CG  + \zeta (1, \eta') \cdot \nabla \p_{x_2} \CG \big) \\
\mathcal{F}_{2\psi}
\phi & = -\phi'', \qquad
\mathcal{F}_{2c}
= -\eta''  
\end{split} \ee
\be \label{global:DF3}  \begin{split}
\mathcal{F}_{3\eta}
\zeta & = \partial_{x_2} \langle \mathcal{H}_\eta(\eta)\zeta, \psi \rangle (0), \qquad \mathcal{F}_{3\psi} 
\phi  = \partial_{x_2} \phi_{\mathcal{H}}(0), \qquad 
\mathcal{F}_{3c}
= 1. \end{split} 
\ee

From here on we shall denote by $D\mathcal{F}$ the Fr\'echet gradient
 \[ D\mathcal{F} := (D_\eta \mathcal{F}, D_\psi \mathcal{F}, D_c \mathcal{F}).\]
The next step is to show that $D \mathcal{F}$ is Fredholm.

\begin{lemma}[Fredholm map]  
At each $(\epsilon, \eta, \psi, c) \in \mathcal{O}_\delta$, $D \mathcal{F}(\epsilon; \eta, \psi, c)$ is a Fredholm operator of index 0.   
\label{global:fredholmlemma} \end{lemma}

\begin{proof}  
Fix $(\epsilon, \eta, \psi, c) \in \mathcal{O}_\delta$.  We seek to establish an estimate of the form: for all $(\zeta, \phi, d) \in X$, 
\be
|(\zeta, \phi, d)|_{X}  \lesssim | D_X\mathcal{F}(\epsilon; \eta, \psi, c)(\zeta, \phi, d) |_Y + | (\zeta, \phi, d)|_{Z}, \label{global:fredholmestimate} 
\ee
where $Z$ is a space such that $X$ is compactly embedded into $Z$, and the constant may depend on $(\epsilon, \eta, \psi, c)$.   An immediate consequence of \eqref{global:fredholmestimate} is that $\mathcal{F}$ is semi-Fredholm.  By the connectedness of $\mathcal{O}_\delta$, and the fact that at the point of bifurcation $D\mathcal{F}$ is an isomorphism, we are then able to deduce that $D \mathcal{F}(\epsilon; \eta, \psi, c)$ is Fredholm of index 0. 

Estimate \eqref{global:fredholmestimate} is merely a consequence of elliptic regularity.  To see this, up to constants  to make the mean zero, rewrite $\CF_{1\eta}$ equation of \eqref{global:DF1} as 
\begin{align*} 
\mathcal{F}_{1\eta}(\epsilon; \eta, \psi, c) \zeta &= M_1 \zeta + M_2 \zeta^\prime + M_3  \langle \mathcal{G}_\eta(\eta) \zeta, \psi \rangle + \alpha^2 \kappa_\eta(\eta) \zeta \\
\mathcal{F}_{1\psi}(\epsilon; \eta, \psi, c) \phi &= M_4 \phi^\prime + M_5 \mathcal{G}(\eta) \phi \\
\mathcal{F}_{1c}(\epsilon; \eta, \psi, c) d & = d M_6.  
\end{align*}
Here the multipliers $M_i$ belong to $H^{k-1}(L\BBS^1)$ depending on $(\eta, \psi, c)$, and so the highest order term is that involving the  derivative of the curvature operator $\kappa_\eta$:
\[ 
\kappa_\eta(\eta) \zeta = 3 \frac{\eta' \eta^{\prime\prime}}{\jbracket{\eta^\prime}^5} \zeta^\prime -\frac{1}{\jbracket{\eta^\prime}^3} \zeta^{\prime\prime}. 
\]
In light of the estimates in Lemma \ref{appendix:propGlemma}, we may treat this as an elliptic problem:
\[
- \zeta'' = \frac 1{\alpha^2} {\jbracket{\eta^\prime}^3} D\CF_{1} (\zeta, \phi, d)  + \mathcal{M}_1(\zeta, \phi, d)
\]
where $\mathcal{M}_1 : Z \to H^{k-2}(L\BBS^1)$ is a bounded operator and $Z := H^{k-1}(L\BBS^1) \times H^{k-1}(L\BBS^1) \times \mathbb{R}$ (which is indeed compactly embedded in $X$.).  From this we have
\begin{align}
| \zeta|_{H^k(L\BBS^1)} & \lesssim |\zeta|_{L^2(L\BBS^1)} + | D_X \mathcal{F}_{1}
(\zeta, \phi, d) |_{H^{k-2}(L\BBS^1)} + | \mathcal{M}_1(\zeta, \phi, d) |_{H^{k-2}(L\BBS^1)} \nonumber  \\
& \lesssim | D_X \mathcal{F}_{1}
(\zeta, \phi, d) |_{H^{k-2}(L\BBS^1)} + | (\zeta, \phi, d) |_{Z}. 
\label{global:fredholmzetabound} \end{align}

Next we turn to control of $\phi$.  From \eqref{global:DF2} we have 
\begin{align*} 
-\phi''& = \mathcal{F}_{2\eta }( \epsilon; \eta, \psi, c) \zeta +  \mathcal{F}_{2\psi}( \epsilon; \eta, \psi, c) 
\phi + \mathcal{F}_{2c}( \epsilon; \eta, \psi, c)d \\
& \qquad + c\zeta'' 
+ \epsilon \p_{x_1} \big((0, \zeta^\prime) \cdot \nabla \CG  - \ep \zeta (1, \eta') \cdot \nabla \p_{x_2} \CG\big)
+ \eta'' d \\
& =: \mathcal{F}_{2\eta }( \epsilon; \eta, \psi, c) \zeta +  \mathcal{F}_{2\varphi}( \epsilon; \eta, \psi, c) \phi + \mathcal{F}_{2c}( \epsilon; \eta, \psi, c)d + \mathcal{M}_2(\zeta, d).  
\end{align*}
It implies 
\[ 
| \phi |_{H^{k} (L\BBS^1)} \lesssim | D_X \mathcal{F}_2(\eta, \psi, c)( \zeta, \phi, d) |_{H^{k-2}(L\BBS^1)} + | \mathcal{M}_2(\zeta, d) |_{H^{k-2}(L\BBS^1)}. 
\] 
Recalling the definition of $\mathcal{M}_2$, we have in fact that 
\[ | \mathcal{M}_2( \zeta, d) |_{H^{k-2}(L\BBS^1)} \lesssim | \zeta |_{H^k(L\BBS^1)} + |d|,\]
and thus \eqref{global:fredholmzetabound} implies
\be \begin{split} | \phi |_{H^{k}(L\BBS^1)} & \lesssim | D_X \mathcal{F}_{1}
(\zeta, \phi, d) |_{H^{k-2}(L\BBS^1)} + | D_X \mathcal{F}_{2}
(\zeta, \phi,d) |_{H^{k-2}(L\BBS^1)} 
+ |(\zeta, \phi, d) |_Z. 
\end{split}  \label{global:fredholmphibound} \ee
Lastly, in view of \eqref{global:DF3},
\[ d = D_X\mathcal{F}_{3}(\epsilon; \eta, \psi, c) (\zeta, \phi, d) - \partial_{x_2} \langle \mathcal{H}_\eta(\eta)\zeta, \psi \rangle (0)
- \partial_{x_2} \phi_{\mathcal{H}}(0),
\]
whence, 
\[ |d| \lesssim |D_X \mathcal{F}_{3}(\epsilon; \eta, \varphi, c) (\zeta, \phi, d)| + |\zeta|_{H^{k-1}(L\BBS^1)} + | \phi |_{H^{k-1}(L\BBS^1)}.\]
This, together with \eqref{global:fredholmzetabound}--\eqref{global:fredholmphibound}, furnishes \eqref{global:fredholmestimate} and thereby the lemma.  
  \end{proof}
  
\begin{lemma}[Spectral properties] \label{global:spectrallemma} The following statements hold.\begin{itemize} 
\item[(i)] For all $\delta > 0$ and $M, \, \theta_0 > 0$, there exist constants $c_1, c_2, c_3> 0$ such that, for all $(\epsilon, \eta, \psi, c) \in \mathcal{O}_\delta$ with $|\epsilon| + | (\eta, \psi, c)|_{X} \leq M$, and for all $(\zeta, \phi, d) \in X$, $\mu \in \mathbb{C}$ with $|\arg{\mu}| \geq \theta_0$, $|\mu| \geq c_1$, $k \ge l \ge 2$,
\be\label{global:spectralestimate} \begin{split}
|\zeta|_{H^l(L\BBS^1)} + |\phi|_{H^l(L\BBS^1)}& \le c_2 |\mu|^{{l-2}/2}  \big(|A_1(\mu) \tilde \psi |_{H^{l-2}(L\BBS^1)} + |A_2(\mu) \tilde \psi |_{H^{l-2}(L\BBS^1)} + |d|\big)\\
&\qquad - c_3 \sum_{j=0}^{l-1} |\mu|^{\frac j2}  \big(|\zeta|_{H^{l-j} (L\BBS^1)} +|\phi|_{H^{l-j} (L\BBS^1)}\big). 
\end{split}\ee
where 
\[ 
A := D_X \mathcal{F}(\epsilon; \eta, \psi, c), \qquad A(\mu) := A - \mu \left(\begin{array}{ccc} 1 &0&0 \\ 0&1&0\\ 0&0&1\end{array}\right),
\]
both of these being viewed as unbounded linear operators from $Y$ to $Y$ with domain of definition $X$ and $A_{1,2}$ are the $\zeta$ and $\phi$ components in $Y$. \\
\item[(ii)] For all $(\epsilon, \eta, \psi, c) \in \mathcal{O}_\delta$, the spectrum $\sigma(A)$ consists of finitely many eigenvalues each of finite algebraic multiplicity, and has no finite accumulation points.  Moreover, there exists an open strip $\mathcal{S}= \mathcal{S}(\epsilon; \eta, \psi, c)$ around the negative real axis $(-\infty, 0]$ in the complex plane such that $\mathcal{S} \cap \sigma(A)$ is finite.  
\end{itemize}  \end{lemma}

\begin{proof} 
(i) Fix $M, \, \theta_0 > 0$, $\mu \in \mathbb{C}$, $|\mu| \ge 1$, and $(\epsilon, \eta, \psi, c) \in \mathbb{R} \times X$ as above.  Let $\xi = \xi(t) : \mathbb{R} \to \mathbb{R}$ be a smooth cutoff function such that 
\[ \textrm{supp }{\xi} \subset [-2, 2], \qquad \xi \equiv 1 \textrm{ on } [-1,1], \qquad \xi \geq 0, \]
and put 
\[ \Psi = \Psi(t, x) := e^{-i|\mu|^{1/2} t} \xi(t) \tilde \psi(x), \qquad \tilde \psi := (\zeta, \phi, d). \]
Let $\mathcal{U} := \mathbb{R} \times L\BBS^1$. An useful observation for such type of functions is 
\be \label{global:MCU}
|e^{-i|\mu|^{1/2} t} \xi(t) f(x)|_{H^l (\mathcal{U})} \lesssim \sum_{m=0}^l |\mu|^{\frac {l-m}2} |f|_{H^m(L\BBS^1)} \lesssim  |e^{-i|\mu|^{1/2} t} \xi(t) f(x)|_{H^l (\mathcal{U})}
\ee
Following Agmon \cite{agmon1962eigenfunctions}, we augmented operator
\[ \mathcal{A} := A + e^{i\theta}  \left(\begin{array}{ccc} \partial_t^2 &0&0 \\ 0&\partial_t^2&0\\ 0&0&0\end{array}\right) ,\]
where $\theta := \arg{\mu}$.  Observe that, for $j=1,2$, 
 \be \label{global:MCA} \begin{split}
 \mathcal{A}_j \Psi & = e^{-i |\mu|^{1/2} t} \xi(t) A_j \tilde \psi - |\mu| e^{i\theta} e^{-i|\mu|^{1/2} t} \xi(t) \tilde \psi_j \\
 & \qquad + e^{i\theta} e^{-|\mu|^{1/2} t} \xi^{\prime\prime}(t) \tilde \psi_j - 2i|\mu|^{1/2} e^{i\theta} e^{-i|\mu|^{1/2} t} \xi^\prime(t) \tilde \psi_j \\
 & =  e^{-i |\mu|^{1/2} t} \xi(t) A_j(\mu)\tilde \psi + e^{i(\theta -|\mu|^{1/2}) t} \big(\xi^{\prime\prime}(t) - 2i|\mu|^{1/2}  \xi^\prime(t) \big) \tilde \psi_j,
 \end{split} \ee
where $\tilde \psi_{1,2}$ is understood as $\zeta$ or $\phi$.

We see that the highest-order terms in $\mathcal{A}_{11}$ are 
\[ \alpha^2 \kappa_\eta (\eta)  + e^{i\theta} \partial_t^2 = -\alpha^2 \partial_{x_1} \left( \jbracket{\eta^\prime}^{-3} \partial_{x_1} \right) + e^{i\theta} \partial_t^2,\]
which is a uniformly elliptic operator on $\mathcal{U}$ (as $|\theta| > \theta_0$ and $|\eta|_{H^k (L\BBS^1)} \le M$ with $k \ge 3$).  As in Lemma \ref{global:fredholmlemma}, we split $\mathcal{A}_1$ into the lower-order terms $\mathcal{M}_1$, that are bounded in $H^{k-1}(\mathcal{U})$, and the higher-order terms above:
\be 
\mathcal{A}_1 \Psi = (e^{i\theta} \partial_t^2 - \alpha^2 \partial_{x_1} ( \jbracket{\eta'}^{-3} \partial_{x_1})) \Psi_1 + \mathcal{M}_1 \Psi,\label{global:spectralellipticPsi1}  \ee
where the linear operator $\mathcal{M}_1$ satisfies 
\[
|\mathcal{M}_1 (\zeta, \phi, d) |_{H^l(L\BBS^1)} \lesssim | \zeta |_{H^{l+1}(L\BBS^1)} + |\phi|_{H^{l+1}(L\BBS^1)} + |d|, \qquad k-1 \ge l \ge 0. 
\]
Using \eqref{global:MCU}, \eqref{global:MCA}, and the elliptic estimates from \eqref{global:spectralellipticPsi1} yields, for $k\ge l \ge 2$
\begin{align}  
| \Psi_1 |_{H^l (\mathcal{U})} & \lesssim | \Psi_1|_{L^2 (\mathcal{U})}+ |e^{-i |\mu|^{1/2} t} \xi(t) A_1(\mu)\tilde \psi|_{H^{l-2}(\mathcal{U})} \nonumber \\
& \qquad + | e^{-i|\mu|^{1/2} t} (\xi^{\prime\prime}(t) - 2i|\mu|^{1/2} \xi^\prime(t)) \zeta |_{H^{l-2}(\mathcal{U})} \nonumber 
+ |\mathcal{M}_1 \Psi|_{H^{l-2}(\mathcal{U})} \nonumber \\
& \lesssim | \zeta |_{L^2(L\BBS^1)}  + |\mu|^{(l-2)/2} | A_1(\mu) \tilde \psi|_{H^{l-2}(L\BBS^1)} \nonumber \\
& \qquad  + \sum_{j=0}^{l-2}  |\mu|^{(j+1)/2} | \zeta |_{H^{l-2-j}(L\BBS^1)} + \sum_{j=0}^{l-2} |\mu|^{j/2}  |\mathcal{M}_1 \tilde \psi|_{H^{l-2-j}(L\BBS^1)}\nonumber\\
& \lesssim |\mu|^{(l-2)/2} | A_1(\mu) \tilde \psi|_{H^{l-2}(L\BBS^1)} + |\mu|^{\frac{l-1}2} |\zeta|_{L^2(L\BBS^1)}+ \sum_{j=0}^{l-2}  |\mu|^{j/2} | \tilde \psi |_{H^{l-1-j}(L\BBS^1)}.
\label{global:spectralestimate1} \end{align}
for $|\mu|$ sufficiently large.   

 
We now repeat this process with the second component of the linearized operator.  
The augmented problem becomes 
\[
\mathcal{A}_2 \Psi = (-\p_{x_1}^2 + e^{i\theta}\p_t^2) \Psi_2 + \mathcal{M}_2 \Psi
\]
where the operator $\mathcal{M}_2$ satisfies 
\[
|\mathcal{M}_2 (\zeta, \phi, d) |_{H^l(L\BBS^1)} \lesssim | \zeta |_{H^{l+2}(L\BBS^1)} + |d|, \qquad k-2 \ge l \ge 0. 
\]
Similarly the ellipticity of $-\p_{x_1}^2+ e^{i\theta}\p_t^2$ and \eqref{global:MCA} imply, for $k\ge l\ge 2$,
\begin{align} 
|\Psi_2|_{H^l (\mathcal{U})} &\lesssim  |\mu|^{{l-2}/2} | A_2(\mu) \tilde \psi |_{H^{l-2}(L\BBS^1)} + \sum_{j=0}^{l-2}  |\mu|^{(j+1)/2} | \phi|_{H^{l-2-j}(L\BBS^1)}\notag\\
&\qquad + \sum_{j=0}^{l-2}  |\mu|^{j/2} (|\zeta|_{H^{l-j} (L\BBS^1)} +|d|). \label{global:MCA2}
\end{align}
First multiplying a large constant to \eqref{global:spectralestimate1}, adding it to \eqref{global:MCA2}, using \eqref{global:MCU}, and then taking $|\mu|>>1$, we obtain, for $k \ge l \ge 2$,
\[ \begin{split}
|\zeta|_{H^l(L\BBS^1)} + |\phi|_{H^l(L\BBS^1)}& \lesssim |\mu|^{{l-2}/2}  \big(|A_1(\mu) \tilde \psi |_{H^{l-2}(L\BBS^1)} + |A_2(\mu) \tilde \psi |_{H^{l-2}(L\BBS^1)} + |d|\big)\\
&\qquad - C \sum_{j=0}^{l-1}  |\mu|^{\frac j2} \big(|\zeta|_{H^{l-j} (L\BBS^1)} +|\phi|_{H^{l-j} (L\BBS^1)}\big). 
\end{split}\]

 
(ii) Note that $A$ is a closed operator and an argument along the lines of Lemma \ref{global:fredholmlemma} shows that $A(\mu)$ is Fredholm.  Since $A = A(0)$ has index 0, invariance of the index tells us that the same must be true for all $\mu$.   To establish that $\mu$ is in the resolvent set of $A$, therefore, it enough to check that $A(\mu)$ is injective.  In fact, estimate \eqref{global:spectralestimate} implies that $A(\mu)$ is injective when $\mu$ satisfies the hypotheses of part (i).   To see this, note that if $(\zeta, \phi, d)$ is in the null space of $A(\mu)$, where $\mu$ is such that \eqref{global:spectralestimate} holds, then for $k \ge l\ge 2$,
\be \label{global:esti1}
|\zeta|_{H^l(L\BBS^1)} + |\phi|_{H^l(L\BBS^1)}\le c_2 |\mu|^{{l-2}/2} |d| - c_3 \sum_{j=0}^{l-1}  |\mu|^{\frac j2} \big(|\zeta|_{H^{l-j} (L\BBS^1)} +|\phi|_{H^{l-j} (L\BBS^1)}\big). 
\ee
From the definitions, we have 
\[ 
0 = A_3 (\mu) (\zeta, \phi, d)  = \partial_{x_2} \langle \mathcal{H}_\eta(\eta)\zeta, \psi \rangle (0) + \partial_{x_2} \phi_{\mathcal{H}}(0) + (1-\mu) d 
\]
which implies 
\[
|d| \lesssim \frac 1{|\mu|} (|\zeta|_{H^2 (\BBS^1)} + |\phi|_{H^2 (\BBS^1)}). 
\]
Substituting this into \eqref{global:esti1} and we obtain, for $k \ge l \ge 2$,
\[
|\zeta|_{H^l(L\BBS^1)} + |\phi|_{H^l(L\BBS^1)}\le C |\mu|^{\frac{l-4}2} (|\zeta|_{H^2 (\BBS^1)} + |\phi|_{H^2 (\BBS^1)})  - c_3 \sum_{j=0}^{l-1}  |\mu|^{\frac j2} \big(|\zeta|_{H^{l-j} (L\BBS^1)} +|\phi|_{H^{l-j} (L\BBS^1)}\big).
\]
Taking $l=2$ we obtain that $\zeta$, $\phi$, and $d$ has to vanish when $|\mu|>>1$
and hence the null space of $A(\mu)$ is trivial. We conclude that
\[ \{ z : |\arg{z}| \geq  \theta_0, |z| > c_2\} \subset \sigma(A)^c.\]
Furthermore, the compactness of the embedding of $X\subset\subset Y$  means that $A$ has a compact resolvent, so that the spectrum is discrete, has finite multiplicity, and no finite accumulation points.  These facts imply the second statement of lemma.    \end{proof}

\subsection{Proof of main result}

\begin{proof}[Proof of Theorem \ref{global:globalift}]  
Suppose there exists $\delta_0 > 0$ such that $\mathscr{C} \subset \mathcal{O}_{\delta_0}$.  Then Lemma \ref{global:propermaplemma}, Lemma \ref{global:fredholmlemma}, and Lemma \ref{global:spectrallemma} show that $\mathcal{F}$ is admissible in the sense of Definition \ref{global:admissibledef} and thus  Theorem \ref{global:kielhofertheorem} applies for all $\delta \in (0, \delta_0)$. Taking $\delta = \frac 12 \delta_0$, we see that in alternative (i) of Theorem \ref{global:kielhofertheorem}, the solution curve can not approach the boundary of $\BBR \times \mathcal{O}_\delta$ it it has to stay in $\BBR \times \mathcal{O}_{\delta_0}$, therefore it implies alternative (i) of Theorem \ref{global:globalift}. If alternative (ii) of Theorem \ref{global:kielhofertheorem} happens, we claim it implies (ii) of Theorem \ref{global:globalift}. If otherwise happens, namely, there does not exists nontrivial zero point of $\CF(\ep=0, \cdot)$, then $\mathscr{C}\setminus \{(0,0,0,0)\} = \mathscr{C}_+ \cup \mathscr{C}_-$, where $\mathscr{C}_\pm := \mathscr{C} \cap (\BBR^\pm \times \mathcal{O})$. Clearly $\mathcal{C}_\pm$ is open in $\mathscr{C}\setminus \{(0,0,0,0)\}$ and thus $\mathscr{C}\setminus \{(0,0,0,0)\}$ is not connested. 

However, if there exists a sequence of $\delta_n > 0$ with $\delta_n \to 0$ for which $\mathscr{C}$ is not a subset of $\mathcal{O}_{\delta_n}$ for any $n$,  then there must be a sequence $\{ (\epsilon_n, \eta_n, \varphi_n, c_n)\} \subset \mathscr{C}$ and $\{x_n\} \in \mathbb{R}$ such that
\[ |(x_n, 1+\eta_n(x_n))|\le \delta_n.\]
Clearly $x_n \to 0$. 
If we assume that alternative (i) does not occur, then $\sup_n |\eta_n^\prime|_{L^\infty} =: M < \infty$.  A simple calculation shows 
\begin{align*} | (0, 1+\eta_n(0))| & \leq \left( |x_n|^2 + |\eta(x_n) - \eta(0)|^2 \right)^{1/2} +\delta_n.\\ 
& \leq |x_n| \jbracket{M} + \delta_n.\end{align*}
It follows that $\eta_n(0) \to -1$ as $n \to \infty$, which implies the third alternative (iii) holds.  \end{proof}

 \appendix 
 \section{Regularity lemmas}
 Here we recall some basic facts about the regularity of the Dirichlet-to-Neumann and harmonic extension operators.   Let us begin by defining the spaces in which we shall obtain estimates.   For any domain $\Omega \subset \mathbb{R}^2$ with $H^s$ boundary, $s > 3/2$, let
 \[ 
 | f |_{H^s(\Omega)} := \inf \{ | F|_{H^s(\mathbb{R}^2)} : F \in H^s(\mathbb{R}^2), ~ F|_{\Omega} = f\}.
 \]
 The homogeneous norm $| \cdot |_{\dot{H}^s(\Omega)}$ is defined analogously.  The spaces $H^s(\Omega)$ and $\dot{H}^s(\Omega)$ are then defined as the quotient of $H^s(\mathbb{R}^2)$ and $\dot{H}^s(\BBR^2)$, respectively with equivalence classes determined by the above norm.   Likewise, we denote
 \[ 
H_{\textrm{m}}^s(\Omega) := \{ f \in H^s(\Omega) : \int_{\partial \Omega} f ds = 0\}, \qquad 
  H_{\textrm{e}}^s(\Omega) := \{ f \in H_{\textrm{m}}^s(\Omega) : \textrm{$f$ is even in $x_1$}\} 
 \]
if $\Omega$ is bounded or symmetric in $x_1$, respectively.  The spaces 
$H_{\textrm{m}}^s(L\BBS^1)$, $H_{\textrm{e}}^s(\mathbb{R})$ are defined in the natural way.  Using appropriately chosen local coordinates, it is possible to define analogous spaces for functions defined on $\partial \Omega$ (cf. \cite[\S6]{shatah2008geometry}). 
 
Consider a domain $\Omega$ given by 
 \[ \Omega = \{ x  : x_2 < 1 + \eta(x_1) \}, \qquad \eta \in H^s(\mathbb{R}) \textrm{ or } H^s(L\BBS^1).\]
 We define the harmonic extension map $\mathcal{H}(\eta)$ by
 \[ \mathcal{H}(\eta) \varphi := \varphi_{\mathcal{H}},\]
 where
 \[ \Delta \varphi_{\mathcal{H}} = 0 \textrm{ in } \Omega, \qquad \varphi_{\mathcal{H}}|_{\partial \Omega} = \varphi.\]
 Notice that this relies on the existence of a well-defined trace operator in $\mathcal{L}(H^{s+1/2}(\Omega),H^s(\partial\Omega))$; the reader is again directed to \cite{shatah2008geometry} for a proof of this fact.  The Dirichlet-to-Neumann operator $\mathcal{N}(\eta)$ is then defined by
\[ 
\mathcal{N}(\eta) \varphi := N \cdot (\nabla \varphi_\mathcal{H})|_{\partial \Omega}.
\]
The non-normalized Dirichlet-to-Neumann operator $\mathcal{G}(\eta)$ is also often used. For any $\psi$ defined on $\BBR$, let $\tilde \psi \big(x_1, \eta(x_1)\big) = \psi(x_1)$ defined on $\p \Omega$. Then 
\[
\big(\mathcal{G} (\eta) \psi \big) (x_1) = \langle \eta' \rangle (x_1) \big( \mathcal{N}(\eta) \tilde \psi\big) (x_1, \eta(x_1)) = \big( \mathcal{N}(\eta) \tilde \psi\big) (x_1, \eta(x_1))\frac {ds}{dx_1}
\]
where $\langle \eta' \rangle =  \big(\sqrt{1+ (\eta^\prime)^2}$. $\mathcal{G}$ is defined similarly for $\eta$ defined on $L\BBS^1$. The next lemma lists the regularity properties of these operators that we require.  

\begin{lemma}\label{appendix:propGlemma}  For $\eta \in H^{s_0}$, $s_0 > 3/2$, let $\dot H^s :=\dot H^s(\mathbb{R})$ or $H_m^s(L\BBS^1)$.  Then 
the following statements are true.    
\begin{itemize}
\item[(a)] The harmonic extension map $\mathcal{H}(\eta) \in \mathcal{L}( H^s (\p \Omega), H^{s+1/2}(\Omega)) \cap \mathcal{L}(\dot H^s (\p \Omega), \dot H^{s+\frac 12} (\Omega))$, $s\in (0, s_0]$. Moreover, its norm is bounded uniform in $\eta$ in a bounded set in $H^{s_0}$. 
\item [(b)] $\mathcal{N}(\eta) \in \mathcal{L} (\dot H^s (\p \Omega), \dot H^{s-1} (\p \Omega))$ and $\mathcal{G}(\eta) \in 
\mathcal{L} (\dot H^s, \dot H^{s-1})$, $s\in [1-s_0, s_0]$, are self-adjoint and invertible. Moreover, their norm and the norm of their inverses are bounded uniform in $\eta$ in a bounded set in $H^{s_0}$. 
\item[(c)] $\mathcal{G}(\eta)$ is $C^\infty$ in their dependence on $\eta$.  For example, the Fr\'echet derivative of $\mathcal{G}(\eta)$ with respect to $\eta$ has the estimate:  For any $\zeta \in H^{s_1}$, $s_1 (\frac 32, s_0]$ and $s \in [1-s_1, s_1]$, 
\[ 
| \langle \mathcal{G}_\eta(\eta) \zeta, \varphi \rangle |_{H^{s-1}} \lesssim | \zeta |_{H^{s_1}} | \varphi |_{H^s} ,
\]
where the constant is uniformly controlled by  $|\eta|_{H^{s_0}}$.  
\end{itemize}  \end{lemma}
The proof of these properties is either straightforward or standard.  In particular, the reader is directed to \cite{sulem1999nlsbook} and \cite[\S 6]{shatah2008geometry}.  

 
 \section{Auxiliary lemmas for the vortex patch}
 In this section, we prove some technical lemmas establishing the stated smoothness of the nonlocal operators used in the proof of Theorem \ref{localpatch:bifurcationtheorem} in section \ref{localpatch:section}.
 
 \begin{lemma} \label{appendix:regGammaFtilde}  Suppose that $\beta \in X^{s}$ for $s > 3/2$, where $X^s$ is as defined in \eqref{localpatch:defXs}.  Then $\Gamma \in H^{s+1/2}(B_1)$, $\widetilde{F} \in H^{s+3/2}(B_1)$ and the mappings
 \[ 
 \beta \mapsto \big(\Gamma(\beta)\big) (z) -z \in \mathcal{L}(X^s, H^{s+1/2}(B_1)), \quad \text{ and } \quad \beta \in X^s \mapsto (\widetilde{F},a) \in H^{s+3/2}(B_1) \times \mathbb{R},\]
 is of class $C^{k_0+1}$. 
 \end{lemma}
 \begin{proof}  Let $\beta \in X^s$ be given and take 
 \[ \Gamma = \Gamma(\beta) := \mathcal{H} (\beta + \cos{\theta}) + i \mathcal{C}\mathcal{H} (\beta + \cos{\theta}),\]
 where $\mathcal{H}$ denotes the harmonic extension on $B_1$, and $\mathcal{C}$ the harmonic conjugate on $B_1$.  It follows from the boundedness of these operators that $\Gamma \in H^{s+1/2}(B_1)$; it is easy to show additionally that $\Gamma$ satisfies the assumptions in  \eqref{localpatch:Gammaassumptions}.   
 
 Next consider $\widetilde{F}$.  Recall that the way in which it is defined is as follows:  By the assumptions $\gamma$, we may solve the following problem uniquely
 \[ \Delta \widetilde{G} = | \partial_z \Gamma|^2 \gamma(\widetilde{G}) \textrm{ in } B_1, \qquad \widetilde{G} \in H^1_0(B_1).\]
For $\gamma \in C^N$, $N > s+k_0 +\frac 12$ and $s >\frac 32$, the mapping $(\beta, \widetilde G) \mapsto  | \partial_z \Gamma|^2 \gamma(\widetilde{G})$ is $C^{k_0+1}$ from $X^s \times H^{\frac 32+s} (B_1)$ to $H^{s -\frac 12} (B_1)$. The nondegeneracy assumption in \eqref{localpatch:gammaassumptions}, the Implicit Function Theorem, and the standard regularity theory imply the local existence and uniqueness of $\widetilde{G} \in H^{s+3/2}(B_1)$ and that it is $C^{k_0+1}$ in $\beta$.  By the positivity of $\gamma$, we know that $\Delta \widetilde{G} > 0$, and hence we may define 
\[ 
 a := \int_{B_1} \Delta \widetilde{G} \, dx, \qquad \widetilde{F} = \frac{1}{a} \widetilde{G}.
\]
 Thus $(\widetilde{F}, a)$ solve \eqref{localpatch:Fsemilineareq}.  From this procedure, the lemma follows immediately.
 \end{proof}
 
We remark that this implies that $D$ and $D_0$ are each $H^{s}$ domains, as they are the image of $B_1$ under $\Gamma$ and $\delta\Gamma$, respectively. 
 
In the following, we consider the function $\nabla \CG\big(x_1, 1+ \ep \widetilde \eta(x_1)\big)$. Given $\beta \in X^s$, it induces $\Gamma$ and $\widetilde F$ which in turn lead to $\widetilde \omega$ and $\omega$ as given in  \eqref{localpatch:defftilde} and \eqref{localpatch:tomega}. Finally $\CG$ is defined in \eqref{setup:CG-loc}. 
 
\begin{lemma} \label{appendix:regTheta} Let $s > 3/2$ and $0\le \ep, \delta \ll 1$, the mapping $(\beta, \widetilde \eta) \mapsto \nabla \CG \big(\cdot, 1+\ep \widetilde \eta(\cdot) \big)$ is $C^{k_0+1}$ from any bounded set in $X^s \times H_e^{s+1} (\BBR)$ to $H_o^{s+1} (\BBR) \times H_e^{s+1} (\BBR)$. 
 \end{lemma}
 
\begin{proof}  The vorticity is determined by $(\widetilde \eta, \beta)$ according to  
 \[ 
 \omega(x) = \frac{\epsilon}{\delta^2} \widetilde{\omega}(\frac{x}{\delta}) \mathds{1}_D(x),
 \]
 where, recall, $\widetilde{\omega} = \Delta \widetilde{f}$, and $D = \delta \Gamma(B_1).$   From Lemma \ref{appendix:regGammaFtilde}, we know that $\omega \in H^{s-1/2}(D)$ and $C^{k_0+1}$ in $\beta$ hence, $\omega \in L^2(\mathbb{R}^2) \cap L^\infty(\mathbb{R}^2)$ and is $C^{k_0+1}$ in $\beta$.  
 
By construction, the vortex patch has approximate radius $\delta$ and sits at the origin, while the free surface is a perturbation of the line $\{ x_2 = 1\}$, and the phantom point vortex in the air sits at $(0,2)$. For $\ep$ and $\delta$ sufficiently small, the strip $\mathcal{S} := \mathbb{R} \times (1/2, 3/2)$
intersects neither $D$ nor the point vortex but contains the entire free surface. One may verity through direct computation that, for any $x^\prime, x'' \notin \mathcal{S}$, 
\[
\log |\cdot -x'| -\log |\cdot -x''| \in H^k (\mathcal{S})
\]
for any $k>0$ and is harmonic. Integrating this linear property, it is clear that the function mapping 
 \[ 
 x \in \mathcal{S} \mapsto \CG(x) = \frac{1}{2\pi} \int_{\mathbb{R}^2} \log{|x-x^\prime|}  \frac{\omega(x^\prime)}{\epsilon}\, dx^\prime - \frac{1}{2\pi} \log{|x - 2\mathbf{e}_2|} \in \mathbb{R} 
 \]
is harmonic and belongs to $H^k(\mathcal{S})$ for any $k>0$. Therefore, $\nabla \CG$ satisfies the same property and thus its trace $\nabla \CG(x_1, 1+\ep \widetilde \eta(x_1))$, for $\widetilde \eta \in H_e^{s+1}$  is in $H^{s+1} (\BBR)$. The symmetry in $x_1$ and the $C^{k_0+1}$ smoothness of $ \nabla \CG \big(\cdot, 1+\ep \widetilde \eta(\cdot) \big)$ follow from the symmetry of $\omega$ and $\widetilde \eta$ and the smoothness of $\widetilde F$ in $\beta$. 
\end{proof}

 \begin{lemma} \label{appendix:regPsi}  Let $s > 3/2$, and define $\widetilde{H}$ by \eqref{localpatch:defPsitilde}.   Then $\widetilde{H}$ is of class $C^{k_0+1} (X^s;H^{s}(\BBS^1))$.  \end{lemma}
 
 \begin{proof}  
 Recall that 
 \begin{align*} 
 \widetilde{H}(\beta) &= \frac{1}{2\pi} a\int_{B_1} \log{|\Gamma(\cdot) - \Gamma(z)|} \gamma(\frac{1}{a} \widetilde{F}(z)) |\Gamma^\prime(z)|^2 \, dz \\
 & = \frac{1}{2\pi} \int_{B_1} \log{|\Gamma(\cdot) - \Gamma(z)|} \Delta \widetilde{F}(z)\, dz.\end{align*}
 In Lemma \ref{appendix:regGammaFtilde} it was shown that $\Gamma \in H^{s+1/2}(B_1)$, $\widetilde{F} \in H^{s+3/2}(B_1)$. 
 
In the following we write $\widetilde H(\beta)$ in a different form. 
Since $\widetilde{\omega} = \Delta \widetilde{f}$, 
 \[ \frac{1}{2\pi} \log{|\cdot|} * \widetilde{\omega} - \widetilde{f}  \qquad \textrm{is harmonic in } D_0.\]
 Denoting by $\mathcal{N}_{D_0}$ the Dirichlet-to-Neumann operator on $D_0$, since $\widetilde f|_{\p D_0} =0$, it follows merely from the definition that
 \[ \mathcal{N}_{D_0} [ (\frac{1}{2\pi} \log{|\cdot|} * \widetilde{\omega})\big|_{\partial D_0} ] =\nabla_N [ \frac{1}{2\pi} \log{|\cdot|} * \widetilde{\omega} - \widetilde{f}],\]
 where $\nabla_N [ \cdot] $ denotes the inner product of the outward unit normal on $\partial D_0$ with the trace of  $\nabla[\cdot]$ on $\partial D_0$.  
 
 Similarly, as $\widetilde{\omega}$ is supported in $D_0$,
 \[ \frac{1}{2\pi} \log{|\cdot|} *\widetilde{\omega} - \frac{1}{2\pi} \log{|\cdot|} \in 
 \dot{H}^1(D_0^c) \qquad \textrm{is harmonic in $D_0^c$.}\]
 Then, writing $\mathcal{N}_{D_0^c}$ for the Dirichlet-to-Neumann operator relative to $D_0^c$, we have
 \[ \mathcal{N}_{D_0^c} [ (\frac{1}{2\pi} \log{|\cdot|} *\widetilde{\omega} - \frac{1}{2\pi} \log{|\cdot|})\big|_{\partial D_0} ] = - \nabla_N [ \frac{1}{2\pi} \log{|\cdot|} *\widetilde{\omega} - \frac{1}{2\pi} \log{|\cdot|} ].\]
 Note that $\nabla_N$ here still refers to the outward unit normal for $D_0$, hence the minus sign. In the above, we have used that fact that $\widetilde \omega \in L^2 (\BBR^2)$ implies $\log |\cdot| * \widetilde \omega \in \dot H^2 (\BBR^2)$ and thus $\nabla \big( \log |\cdot| * \widetilde \omega  \big)|_{\p D_0} \in H^{\frac 12} (\p D_0)$ is well-defined. Summing these two identities yields the following 
 \begin{align*}
 \mathcal{N}_{D_0} (\frac{1}{2\pi} \log{|\cdot|} * \widetilde{\omega})|_{\partial D_0} + \mathcal{N}_{D_0^c}( \log{|\cdot|} * \widetilde{\omega} - \frac{1}{2\pi} \log{|\cdot|} )|_{\partial D_0} & =   \nabla_N \frac{1}{2\pi} \log{|\cdot|} - \nabla_N \widetilde{f}.  \end{align*}
 In order to relate this to $\widetilde{H}$, let us introduce the operator $\widetilde{\mathcal{N}}(\beta): H^{s}(S^1) \to H^{s-1}(S^1)$ by 
 \[ \widetilde{\mathcal{N}}(\beta)(g) := [\mathcal{N}_{D_0} (g \circ \Gamma(\beta)^{-1})] \circ \Gamma.\]
Much as in Lemma \ref{appendix:propGlemma}, $\widetilde{\mathcal{N}}(\beta)  \in \mathcal{L}(H^{s^\prime}(S^1), H^{s^\prime-1}(S^1))$, for all $s^\prime \in [1-s,s]$, it is $C^\infty$ in $\beta$ with uniform bounds for $\beta$ in any bounded set in $X^s$.  For the complimentary domain, $D_0^c$, we introduce the analogous operator  $\widetilde{\mathcal{N}}^c(\beta)  \in \mathcal{L}(H^{s^\prime}(S^1), H^{s^\prime-1}(S^1))$ defined by
 \[ \widetilde{\mathcal{N}}^c(\beta)(g) :=  [\mathcal{N}_{D_0^c} (g \circ \Gamma(\beta)^{-1})] \circ \Gamma.\]
 Lastly, to pullback  the normal derivative on $D_0$ to $B_1$, let 
 \[ \widetilde{\nabla}_N =  \widetilde{\nabla}_N(\beta): H^s(B_1) \to H^{s-3/2}(S^1), \] 
 be the operator
 \[ \widetilde{\nabla}_{N(\beta)}(g) := [\nabla_N(g \circ \Gamma(\beta)^{-1})] \circ \Gamma(\beta),\]
 where $\nabla_N$ is the outward unit normal derivative on $D_0$.  
 Then, because $\widetilde{H} \circ \Gamma^{-1} = \frac{1}{2\pi}\log{|\cdot|} * \widetilde \omega$, and $\widetilde{f} = \widetilde{F} \circ \Gamma^{-1}$, with some shuffling of terms we have
\begin{align*}
  (\widetilde{\mathcal{N}} + \widetilde{\mathcal{N}}^c) ( \widetilde{H} -\frac{1}{2\pi} \log{|\Gamma|} )& =  \widetilde{\nabla}_{N}( \frac{1}{2\pi} \log{|\Gamma|} -  \widetilde{F}) - \widetilde{\mathcal{N}} \frac{1}{2\pi} \log{|\Gamma|}.  \end{align*}
Since the right side of the above equation belongs to $H^{s-1} (\BBS^1)$, it is clear that $\widetilde H \in H^{s} (\BBS^1)$. Moreover, the smoothness of $\mathcal{N}$ in $\Gamma$ and the smoothness of $\Gamma$ (and thus that of $N$) and $\widetilde F$ in $\beta$ imply the $C^{k_0+1}$ smoothness of $\widetilde H$ in $\beta$.   
\end{proof}

  Putting together these lemmas, we arrive at the following.
  
 \begin{lemma} \label{appendix:regF}  Let $s > 3/2$, and define $\mathcal{F} = (\mathcal{F}_1, \mathcal{F}_2, \mathcal{F}_3)$ as in \eqref{localpatch:defF}.  Then $\mathcal{F}$ is well-defined in the sense that $\mathcal{F}(\epsilon, \delta; \eta, \varphi, c, \beta) \in \mathcal{Y}$, for $(\epsilon, \delta; \eta, \varphi, c, \beta) \in \mathbb{R}^2 \times \mathcal{X}$.  Moreover, $\mathcal{F} \in C^{k_0+1} (\mathbb{R}^2\times \mathcal{X}; \mathcal{Y})$.
 \end{lemma}
 
 \subsection*{Notation}  
\begin{itemize}  
\item For any scalar quantity $a$, we set $\langle a \rangle := \sqrt{1+a^2}$.
\item $\varphi_{\mathcal{H}}$ is the harmonic extension of $\varphi$ to $\Omega(\eta)$.  When we wish to make the dependence on $\eta$ more explicit, we shall write $\mathcal{H}(\eta)\varphi$.   
\item  $\mathcal{G}(\eta) = \jbracket{\eta^\prime}\mathcal{N}(\eta)$, $\mathcal{N}(\eta)$ being the Dirichlet-to-Neumann operator in the fluid domain 
  \[ \Omega = \Omega(\eta) := \{ (x_1,x_2) : x_2 < \eta(x_1) \}\]
 with vanishing at $x_2 = -\infty$.   For regularity properties of $\mathcal{G}(\eta)$, see Lemma \ref{appendix:propGlemma}.  
 \item For a region $\Omega \subset \mathbb{R}^2$, we denote the outward unit normal on $\partial \Omega$ as $N$, and the outward normal derivative by $\nabla_N$.  
 \item $\kappa$ is the curvature of the free surface; viewed as a differential operator on $\eta$ it is
 \[ \kappa(\eta) = -\partial_{x_1} \left( \frac{\eta^\prime}{\jbracket{\eta^\prime}} \right).\] 
 \item If $X$ and $Y$ are Banach spaces and $F = F(\lambda, x) : \mathbb{R} \times X \to Y$ is Fr\'echet differentiable with respect to $x$, we denote its Fr\'echet derivative evaluated at $(\lambda_0, x_0)$ by $F_x(\lambda_0, x_0)$ or $D_x F(\lambda_0, x_0)$.
 \item For Banach spaces $X$, $Y$, $\mathcal{L}(X;Y)$ designates the space of bounded linear operators from $X$ to $Y$.
 \item If $X, Y$ are Banach spaces and $L$ is a linear map from $X$ to $Y$, we write $\mathscr{N}(L)$ for the null space of $L$, and $\mathscr{R}(L)$ for the range.  
 \item When $X_1$, $X_2$, and $Y$ are Banach spaces and $G = G(f,g) : X_1 \times X_2 \to Y$ is Fr\'echet differentiable with respect to $f$, then we let 
 \[ \langle G_f(f_0, g_0) f, g \rangle \]
 denote the Fr\'echet derivative of $G$ in the direction $f$, evaluated at $(f_0, g_0)$, and applied to $g$.  
 \end{itemize}

\bibliographystyle{siam}
\bibliography{steadywaves}

\end{document}